\documentclass[11pt]{article}
\usepackage{amsmath,amsfonts,amssymb,amsthm}
\usepackage{mathtools}

\usepackage{colortbl}

\RequirePackage[colorlinks,citecolor=blue,urlcolor=blue,backref=page]{hyperref}

\newtheorem{proposition}{Proposition}[section]
\newtheorem{theorem}[proposition]{Theorem}
\newtheorem{corollary}[proposition]{Corollary}
\newtheorem{lemma}[proposition]{Lemma}
\newtheorem{remark}[proposition]{Remark}

\newtheorem{example}[proposition]{Example}

\newcommand{\nc}{\newcommand}
\nc{\I}{{\mathbf 1}}
\nc{\bN}{{\mathbf N}}
\nc{\bG}{{\mathbf G}}
\nc{\bM}{{\mathbf M}}
\nc{\cB}{{\mathcal B}}
\nc{\cM}{{\mathcal M}}
\nc{\R}{{\mathbb R}}
\nc{\N}{{\mathbb N}}
\nc{\Z}{{\mathbb Z}}
\nc{\BX}{{\mathbb X}}
\nc{\BY}{{\mathbb Y}}
\nc{\BM}{{\mathbb M}}
\nc{\cX}{{\mathcal X}}
\nc{\cY}{{\mathcal Y}}
\nc{\cN}{{\mathcal N}}
\nc{\cF}{{\mathcal F}}

\DeclareMathOperator{\supp}{supp}

\nc{\BP}{\mathbb{P}}
\nc{\BE}{\mathbb{E}}
\nc{\BQ}{\mathbb{Q}}
\nc{\BU}{\mathbb{U}}

\numberwithin{equation}{section}

\begin{document} 

\renewcommand{\thefootnote}{\fnsymbol{footnote}}
\author{{\sc {\sc Mikhail Chebunin\footnotemark[1]} \hspace{0.1mm} and G\"unter Last\footnotemark[2]} \\ {\small Karlsruhe Institute of Technology, Institute of Stochastics, 76131 Karlsruhe, Germany.}}
\footnotetext[1]{chebuninmikhail@gmail.com
}
\footnotetext[2]{guenter.last@kit.edu
}

\title{On the uniqueness of the infinite cluster and the cluster density\\ in the Poisson driven random connection model} 
\date{\today}
\maketitle

\begin{abstract} 
\noindent 
We consider a random connection model  (RCM) on a general space
driven by a Poisson process whose intensity measure is scaled by a parameter $t\ge 0$.
We say that the infinite clusters are deletion stable if the removal of a Poisson
point cannot split a cluster in two or more infinite clusters.
We prove that this stability together with a natural irreducibility assumption implies uniqueness of the infinite
cluster.  Conversely, if the infinite cluster is unique then this stability property holds.
Several criteria for irreducibility will be established.
We also study the analytic properties of expectations of functions of clusters
as a function of $t$. In particular we show that the position dependent cluster density
is differentiable.  A significant part of this paper is devoted to 
the important case of a stationary marked RCM (in Euclidean space), 
containing the Boolean model with general compact grains and
the so-called weighted RCM as special cases.
In this case we establish differentiability
and a convexity property of the cluster density $\kappa(t)$. These properties
are crucial for our proof of deletion stability
of the infinite clusters but are also of interest in their own right.
It then follows that an irreducible stationary marked RCM can have at most one infinite cluster. 
This extends and unifies several results in the literature.
 \end{abstract}


\vspace{0.1cm}

\section{Introduction}\label{intro}

Let $\BX$ be a complete separable metric space, 
denote its Borel-$\sigma$-field by $\cX$, and let $\lambda$ be
a locally finite and diffuse measure on $\BX$.
Let $t\in\R_+:=[0,\infty)$ be an intensity parameter
and let $\eta$ be a Poisson process on $\BX$ with intensity
measure $t\lambda$, defined over a probability space 
$(\Omega,\cF,\BP)$. We often write $\BP_t$ instead
of $\BP$ and $\BE_t$ for the associated expectation operator.

Let $\varphi\colon\BX^2\to[0,1]$ be a measurable
and symmetric function satisfying
\begin{align}\label{evarphiintlocal}
D_\varphi(x):=
\int \varphi(x,y)\,\lambda(dy)<\infty,\quad \lambda\text{-a.e.} \ x\in\BX.
\end{align}
We refer to $\varphi$ as {\em connection function}.
The {\em random connection model} (RCM)
is the random graph $\xi$ whose vertices are the points of $\eta$
and where a pair of distinct points $x,y\in\eta$ 
forms an edge with probability $\varphi(x,y)$, independently
for different pairs.
In an Euclidean setting the 
RCM was introduced in \cite{Pen91}; see \cite{MeesterRoy}
for a textbook treatment. It can be defined on point processes other than Poisson.
The general Poisson version was studied in \cite{LastNestSchul21}.
The RCM is a fundamental and versatile example of a spatial random graph. 
Of particular interest is the {\em stationary marked} case.
In this case we have $\BX=\R^d\times\BM$ for some
mark space $\BM$ and $\lambda$ is proportional to the
product of Lebesgue measure and a given mark distribution.
Then the RCM becomes stationary and ergodic under
shifts in the spatial coordinate.
This model contains the Boolean model (see \cite{LastPenrose18, SW08}) with general compact grains and
the so-called weighted RCM as special cases
and keeps attracting a lot of attention; see e.g.\ 
\cite{BetschLast21, CaicedoDickson24, DicksonHeydenreich22, GHMM22, HHLM23, LastZiesche17, Pen16}.

Following common terminology of percolation theory we refer to a component
of $\xi$ as {\em cluster}. The RCM $\xi$ {\em percolates}, if it has an infinite cluster,
that is a component with infinitely many vertices.
We say that the infinite clusters of $\xi$ are {\em deletion stable} if
the removal of a point cannot split a cluster in two or more infinite clusters.
If the infinite cluster is unique, then it is easy to prove that $\xi$ is 
deletion stable.
In fact, $\xi$ is then almost surely even {\em $2$-indivisible} in the sense of \cite{NashWilliams};
see Corollary \ref{c:necess}
Our first main result (Theorem \ref{t:unique})
says that deletion stability together with {\em irreducibility} 
implies (almost sure) uniqueness of the infinite cluster.
We prove this by a peculiar addition and removal procedure, which seems
to be new. Our method crucially relies on the properties of the underlying Poisson process.
Irreducibility is a very natural assumption for uniqueness (see Remark \ref{r:nonunique}) and
will be discussed in Section \ref{sec:irred}.
Theorem \ref{tmain1} shows that the infinite clusters of the stationary marked RCM are deletion stable. 
This is the second main result of this paper. Our proof transfers some of the beautiful ideas from the seminal paper \cite{Aizenman87} 
by Aizenman, Kesten and Newman to the continuum.  
To this end we  significantly extend and complement the arguments in \cite{Jiang11}, where the methods from \cite{Aizenman87}
were used to treat the Gilbert graph with deterministic balls.
Taken together,  Theorems \ref{t:unique}  and \ref{tmain1} yield uniqueness of the infinite cluster of an irreducible stationary marked RCM;
see Theorem \ref{t:markedunique}. This extends and unifies several results in the literature.
The stationary (unmarked) RCM was treated in \cite{MeesterRoy} for an isotropic and norm-decreasing connection function;
see also \cite{Alexander95b}. 
A special case of the marked RCM was treated in \cite{JaMoe17}. The uniqueness of the infinite cluster of the spherical Boolean model
was proved in \cite{MeesterRoy94,MeesterRoy}.
As a consequence of 
we also obtain that an irreducible stationary marked RCM is $2$-indivisible.

We also establish several analytic properties of
cluster expectations, first in the general and then in the stationary marked case.  
Since clusters are not locally determined, the proof of these results requires some efforts. 
In particular we show that the {\em position dependent cluster density}  (given by \eqref{FE}) 
is, as a function of $t$,  continuously differentiable; see Theorem \ref{t:F_deriv}.
In the stationary marked case this is true
for the {\em cluster density} $\kappa(t)$, defined by \eqref{e:markedfree}; see Theorem \ref{tdiffenergy}. Our 
proofs partially follow  \cite{BeGrimLoff98}, where 
the Boolean model with deterministic balls was considered.
We also prove that $t\kappa(t)+d_\varphi t^2/2$ is a convex function of $t$,
where $d_\varphi$ is the expected degree of a typical vertex, given by \eqref{e:intmarked}.
This remarkable property was established in \cite{Aizenman87,DurrettNguyen85} for discrete percolation models 
and in \cite{Jiang11} for the Boolean model with deterministic balls. 
This convexity is crucial for proving deletion stability of the infinite clusters in the stationary marked case
and its proof heavily depends on the (amenability) properties of Euclidean space; see Remark \ref{r:amenability}.

With the exception of \cite{Jiang11}, all previous uniqueness proofs in continuum
percolation seem to use the approach in \cite{BurtonKeane89}; see also 
\cite{GandKeaneNewman92}. 
It is often argued that this approach is more elegant than
the one in \cite{Aizenman87}. However, our paper shows that the methods 
from \cite{Aizenman87} can be conveniently extended to the continuum,
at least in the case of a Poisson driven RCM. Moreover, this approach provides
a lot of additional information on the clusters, which are valid
for all values of the intensity parameter $t$. And last but not least
our general uniqueness theorem applies to a general state space $\BX$,
without {\em any} structural assumptions.

The paper is organized as follows. In Section \ref{s:definitions} we
give the formal definition of the RCM $\xi$, while Section
\ref{s:basics} presents the RCM version of the multivariate Mecke
equation and the Margulis--Russo formula.  In Section
\ref{s:stationaryRCM} we discuss the stationary marked RCM, an
important special case of the general RCM.  In Section \ref{sec:irred}
we define a RCM to be irreducible if, roughly speaking, every pair of
Poisson points has a positive probability of being in the same
cluster. Without such a property one cannot expect the infinite
cluster (if it exists) to be unique. For a stationary marked RCM 
Theorem \ref{l:77} characterises irreducibility in terms of the symmetric function
$\int \varphi((0,p),(x,q))\,dx$,  
which is (up to the factor $t$) the density of the expected number of neighbours of 
$(0, p)$ with respect to the mark distribution. 
Theorems \ref{t:irredgeneral} and \ref{t:irredPOP} provide sufficient
conditions for irreducibility under more specific assumptions. 
In Section \ref{sec:toleranceuniqueness} we prove that deletion stability of
infinite clusters and irreducibility together imply uniqueness of the
infinite cluster; see Theorem \ref{t:unique}.  Section \ref{secMarkov} presents a spatial Markov
property.  In Section \ref{secperturb} 
we establish differentiability of certain cluster expectations, while
Section \ref{secdiffop} rewrites the derivatives as
Margulis--Russo type formulas. In
Section \ref{s:diffenergy} we show that the position dependent cluster density 
is continuously differentiable. In Section \ref{s:tolRCM} we
prove that the infinite clusters of the stationary marked RCM are
deletion stable; see Theorem \ref{tmain1}. The final Section
\ref{s:markedirrunique} provides several examples of irreducible stationary marked RCMs.

\bigskip

For the reader's convenience, we list below our main results separately for the general and the stationary marked cases.

\bigskip
\noindent
{\bf Main results for the general RCM:}

\begin{itemize}

\item Theorem \ref{t:unique} shows that an irreducible RCM 
with deletion stable infinite clusters can have at most one infinite cluster,
while showing that deletion stability is necessary for uniqueness.

\item Theorem \ref{t:F_deriv} shows continuous differentiability of certain cluster expectations, 
while Theorem \ref{t:MR_2} and Remark \ref{r:MR} rewrite the derivative as a Margulis-Russo type formula.

\item Theorem \ref{tdiffenergy} shows continuous differentiability of
  the position dependent cluster density, while Theorem
  \ref{tdiffenergymarked} shows that this remains true after some additional integration.

\end{itemize}

\noindent
{\bf Main results for the stationary marked RCM:}

\begin{itemize}

\item 
Theorem \ref{l:77} characterises irreducibility while 
Theorems  \ref{t:irredgeneral} and \ref{t:irredPOP} provide sufficient conditions under more specific assumptions.

\item Theorem \ref{tmain1} shows that the infinite clusters of a stationary marked RCM are deletion stable.

\item Theorem \ref{t:convex} 
shows that the cluster density $\kappa(t)$ is continuously differentiable and that
$t\kappa(t)+d_\varphi t^2/2$ is convex.

\end{itemize}

\section{Formal definition of the RCM}\label{s:definitions}

It is convenient to model a RCM as a suitable point process.
Let $\bN$ denote the space of all simple locally finite counting measures on $\BX$, equipped with the standard $\sigma$-field, see e.g.\ \cite{LastPenrose18}.
A measure $\nu\in\bN$ is identified with its support $\{x\in\BX:\nu(\{x\})=1\}$ and describes the set of vertices of a (deterministic) graph.
If $\nu(\{x\})=1$ we write $x\in\nu$.
Using the Dirac measure $\delta_x$ at point $x\in\BX$,
any $\nu\in\bN$ can be written as a finite or infinite sum
$\nu=\delta_{x_1}+\delta_{x_2}+\cdots$, where the $x_i$ are pairwise distinct
and do not accumulate in bounded sets.
The space of (undirected) graphs with vertices from $\BX$ (and no loops) is
described by the set $\bG$ of all counting measures $\mu$ on $\BX\times\bN$
with the following properties. First we assume that the measure $V(\mu):=\mu(\cdot\times\bN)$ 
is locally finite and simple, that is, an element of $\bN$.
  Hence,  if $x\in V(\mu)$ (that is $\mu(\{x\}\times\bN)=1$), then
there is a unique $\psi_x\in\bN$ such that $(x,\psi_x)\in\mu$. We assume that
$x\notin \psi_x$. Finally, if $x\in V(\mu)$ and $y\in\psi_x$ then we assume that
$(y,\psi_y)\in \mu$ and $x\in\psi_y$. Also $\bG$ is
equipped with the standard $\sigma$-field.
There is an edge between $x,y\in V(\mu)$ if $y\in\psi_x$ (and hence $x\in\psi_y$). If $\psi_x=0$, then $x$ is 
{\em isolated}.

We write $|\mu|:=\mu(\BX\times\bN)$ for the cardinality of $\mu\in\bG$
and similarly for $\nu\in\bN$. Hence $|\mu|=|V(\mu)|$. 
For $x,y\in V(\mu)$ we write $x\sim y$ (in $\mu$) if there is an edge between $x$ and $y$
and  $x \leftrightarrow y$ (in $\mu$) if there is a path in $\mu$ leading from $x$ to $y$.
For $A\subset\BX$ we write $x\sim A$ (in $\mu$) if there exists $y\in A\cap V(\mu)$
such that $x\sim y$.

Let $\mu,\mu'\in \bG$. Then $\mu$ is a {\em subgraph} of $\mu'$ if
$V(\mu)\le V(\mu')$ (as measures) and if $(x,\psi)\in \mu$ and
$(x,\psi')\in\mu'$ together imply that  $\psi\le \psi'$. Note that this is not the same
as $\mu\le\mu'$.

Let $\chi$ be a simple point process on $\BX$, that is a random element of
$\bN$. The reader should think of
a Poisson process possibly augmented by additional (deterministic) points.
By \cite[Proposition 6.2]{LastPenrose18} there exist random elements $X_1,X_2,\ldots $ of
$\BX$ such that
\begin{align}\label{e:points}
\chi=\sum^{|\chi|}_{n=1}\delta_{X_n},
\end{align}
where $X_m\ne X_n$ whenever $m\ne n$ and
$m,n\le  |\chi|$.
Let $(Z_{m,n})_{m,n\in\N}$ be a double sequence of random elements 
uniformly distributed on $[0,1]$ such that $Z_{m,n}=Z_{n,m}$ for all
$m,n\in\N$ and such that the $Z_{m,n}$, $m<n$, are independent.
Then the RCM (based on $\chi$) is the point process
\begin{align}\label{e:defxi}
\xi:=\sum^{|\chi|}_{m=1}\delta_{(X_m,\Psi_m)},
\end{align}
where 
\begin{align*}
\Psi_m:=\sum^{|\chi|}_{n=1}\I\{n\ne m,Z_{m,n}\leq\varphi(X_m,X_n)\}\delta_{X_n}.
\end{align*} 
In this notation we suppress the dependence on the $Z_{m,n}$.
While the definition of $\xi$ depends on the ordering of the points of $\chi$,
its distribution does not.

Below we will work with a Poisson process $\eta$ with a diffuse
intensity measure $\lambda$. Then $\eta$ is simple and can be identified with its support.
Otherwise it is not obvious how to treat multiplicities.
One way to proceed is described in the following remark.

\begin{remark}\label{r:non-diffuse}\rm Let $\lambda$ be a possibly non-diffuse locally finite measure on
$\BX$ and let $\eta$ be a Poisson process with intensity measure $\lambda$; see 
\cite{LastPenrose18}. Then $\eta$ is a random element of the space of all locally
finite counting measures on $\BX$ whose atoms may have multiplicities.
Let $\BU$ denote the uniform distribution on $[0,1]$ and let $\hat\eta$ be a
$\BU$-marking of $\eta$. Then $\hat\eta$ is a Poisson process on $\hat\BX:=\BX\times[0,1]$
with intensity measure $\hat\lambda:=\lambda\otimes\BU$; see \cite[Theorem 5.6]{LastPenrose18}. 
Since $\hat\lambda$ is diffuse, $\hat\eta$ is simple and can be written as
\begin{align}\label{e:points2}
\hat\eta=\sum^{\hat\eta(\BX)}_{n=1}\delta_{(X_n,U_n)},
\end{align}
where $(X_m,U_m)\ne (X_n,U_n)$ whenever $m\ne n$.
Given a connection function $\varphi$ we can define
a connection function $\hat\varphi\colon\hat\BX^2\to[0,1]$.
We set $\hat\varphi((x,u),(y,v)):=\varphi(x,y)$ for $u,v\in[0,1]$   if $x\ne y$.
If $x=y$ we set $\hat\varphi((x,u),(x,v)):=1$.
The resulting random connection model $\hat\xi$ admits the following interpretation.
An atom of $\eta$ with size $k\in\N$ is split into $k$ atoms of size $1$.
A pair of atoms of size $1$ sitting at different locations $x$ and $y$ are independently connected
with probability $\varphi(x,y)$. 
A pair of atoms of size $1$ sitting at the same location are always connected. 
By definition of $\hat\varphi$, we have
$C^x:=C^{(x,u)}(\xi_1)$ for each $u\in[0,1]$ with $(x,u)\in\hat\eta$.
Hence $\hat\xi$ induces a random graph $\xi^*$ on the support $\supp\eta$ of $\eta$ with
components $C^x$, $x\in\eta$. The infinite components of this random graph are in one-to-one correspondence
to those of $\xi$.
\end{remark}

\begin{remark}\label{rnon-diffuse2}\rm Establish the setting of Remark \ref{r:non-diffuse}
There are other ways to define a random connection model driven by $\hat\eta$. 
For instance we might set $\hat\varphi((x,u),(x,v)):=0$, while (as before)
$\hat\varphi((x,u),(y,v)):=\varphi(x,y)$ for $x\ne y$.
Then a pair of atoms sitting at the same location is never connected. In our opinion the
choice in Remark \ref{r:non-diffuse} is rather natural.
\end{remark}

We now introduce some notation used throughout the paper.
For $\mu,\mu'\in\bG$ we often interpret $\mu+\mu'$ as the
measure in $\bG$ with the same support as $\mu+\mu'$.
A similar convention applies to $\nu,\nu'\in\bN$.
Let $\mu\in\bG$. For $B\in\cX$ we write $\mu(B):=\mu(B\times \bN)$.
More generally, given a measurable function $f\colon \BX\to \R$ we write 
$\int f(x)\,\mu(dx):=\int f(x)\,\mu(dx\times \bN)$. 
Similarly, given $x\in\BX$, we write
$x\in\mu$ instead of $x\in V(\mu)=\mu(\cdot\times \bN)$.
In the same spirit we write $g(\mu):=g(V(\mu))$, whenever
$g$ is a mapping on $\bN$.
These (slightly abusing) conventions lighten the notation and should not cause any confusion. 
For $B\in\cX$ 
we denote by $\mu[B]\in\bG$ the {\em restriction} of $\mu$ to $B$, that is  the graph
with vertex set $V(\mu)\cap B$ which keeps only those edges
from $\mu$ with both end points from $B$. In the same way we use the
notation $\mu[\nu]$ for $\nu\in\bN$. 
Similarly for a measure $\nu$ on $\BX$ (for instance for $\nu\in\bN$) 
we denote by $\nu_B:=\nu(B\cap\cdot)$  the restriction of $\nu$ to a set $B\in\cX$.
Assume now that $v\in V(\mu)$. For $n\in\N_0$ let
$C^v_n(\mu)\in\bG$ denote the graph restricted to the set of vertices $x\in V(\mu)$ with
$d_\mu(v,x)=n$, where $d_\mu$ denotes distance within the graph $\mu$.  
Note that $C^v_0(\mu)$ is just the isolated vertex $v$. Slightly abusing our notation
we write $C^v_0(\mu)=\delta_v$.  For $v\notin V(\mu)$ we set $C^v(\mu):=0$,
interpreted as an empty graph (a graph with no vertices).
The {\em cluster} $C^v(\mu)$  of $v$ in $\mu$ is the graph $\mu$ restricted to
\begin{align*}
\sum^\infty_{n=0}V(C^v_n(\mu)),
\end{align*}
while $C^v_{\le{n}}(\mu)$, $n\in\N_0$, is the graph $\mu$ restricted to 
$V(C^v_0(\mu))+\cdots+V(C^v_n(\mu))$. 
For later purposes it will be convenient to define $C^v_{\le{-1}}(\mu)=C^v_{-1}(\mu):=0$ as the zero measure.
For $\mu\in\bG$ and $x\in\BX$
we denote by $\mu-\delta_x:=\mu[V(\mu)-\delta_x]$
the graph resulting from $\mu$ by removing the point $x$.
If $x\notin V(\mu)$ then $\mu-\delta_x=\mu$.

\section{Basic properties of the RCM}\label{s:basics}

Let $\xi$ be a RCM based on a Poisson process $\eta$ on 
$\BX$ with diffuse intensity measure $\lambda$.
Our first crucial tool is a version of the Mecke equation (see
\cite[Chapter 4]{LastPenrose18}) for  $\xi$.  Given $n\in\N$ and 
$x_1,\ldots,x_n\in\BX$ we denote
$\eta^{x_1,\ldots,x_n}:=\eta+\delta_{x_1}+\cdots+\delta_{x_n}$ (removing possible multiplicities)
and let $\xi^{x_1,\ldots,x_n}$ denote a RCM based on $\eta^{x_1,\ldots,x_n}$.
It is useful to construct $\xi^{x_1,\ldots,x_n}$ in a specific way as follows.
We connect $x_1$ with the points in $\eta$ using independent connection decisions which
are independent of $\xi$.
We then proceed inductively finally connecting $x_n$ to 
$\eta+\delta_{x_1}+\cdots + \delta_{x_{n-1}}$.
For $n\in\N$ and a measurable function
$f\colon\BX^n \times \bG\to [0,\infty]$  the Mecke equation for $\xi$ states
that 
\begin{align}\label{e:Mecke}
\begin{split}
    \BE \int f(x_1,\ldots,x_n,\xi)\,\eta^{(n)}&(d(x_1,\ldots,x_n)) \\
=& \BE \int f(x_1,\ldots,x_n,\xi^{x_1,\ldots,x_n})\,\lambda^n(d(x_1,\ldots,x_n)),
\end{split}
\end{align}
where integration with respect to  the {\em factorial measure} $\eta^{(n)}$
of $\eta$ means summation over all $n$-tuples of pairwise distinct points
from $\eta$. A convenient way to prove this and related formulas is
to introduce a probability kernel $\Gamma$ from $\bN$ to $\bG$, satisfying
\begin{align}\label{e:kernel}
\BP((\eta,\xi)\in\cdot)=\BE \int \I\{(\eta,\mu)\in\cdot\}\,\Gamma(\eta,d\mu).
\end{align}
The kernel $\Gamma$ is just a regular version of the conditional distribution of $\xi$
given $\eta$ and can be defined explicitly; see Section \ref{s:definitions}.
A crucial property of this kernel is
\begin{align}\label{c1}
\BE\Gamma(\eta^{x_1,\ldots,x_n},\cdot)=\BP(\xi^{x_1,\ldots,x_n}\in\cdot),\quad \lambda^n\text{-a.e.\ $(x_1,\ldots,x_n)\in \BX^n$}.
\end{align}
It follows from \cite[Theorem 4.4]{LastPenrose18}
that the left-hand side of \eqref{e:Mecke} is given by
\begin{align*}
\BE \iint f(x_1,\ldots,x_n,\mu)\,\Gamma(\eta^{x_1,\ldots,x_n},d\mu)\,\lambda^n(d(x_1,\ldots,x_n)).
\end{align*}
Therefore \eqref{e:Mecke} follows from \eqref{c1}.

Given $v\in\BX$ we sometimes use \eqref{e:Mecke} in the form
\begin{align}\label{e:Meckev}
\begin{split}
    \BE \int f(x_1,\ldots,x_n,\xi^v)\,\eta^{(n)}&(d(x_1,\ldots,x_n)) \\
=& \BE \int f(x_1,\ldots,x_n,\xi^{v,x_1,\ldots,x_n})\,\lambda^n(d(x_1,\ldots,x_n)).
\end{split}
\end{align}
This can be derived from \eqref{e:Mecke} as follows. We can write $\xi^v=h(\xi,v,U)$, where $U$ is a random element
of $[0,1]^\N$ with
independent and uniformly distributed components, independent of $\xi$;
see the proof of Lemma \ref{l:remove} for more detail.
It remains to note that $h(\xi^{x_1,\ldots,x_n},v,U)$ has the
same distribution as $\xi^{v,x_1,\ldots,x_n}$, provided that $v,x_1,\ldots,x_n$ are pairwise distinct.

To state another useful version of  $\eqref{e:Mecke}$ we recall the
notation $\mu-\delta_x=\mu[V(\mu)-\delta_x]$ for $\mu\in\bG$ and $x\in\BX$.
Given $n\in\N$ and $x_1,\ldots,x_n\in\BX$ we define
$\mu-\delta_{x_1}-\cdots -\delta_{x_n}$ inductively.
The kernel $\Gamma$ has the property
\begin{align*}
\int \I\{\mu-\delta_{x_1}-\cdots -\delta_{x_n}\in\cdot\}\,\Gamma(\nu,d\mu)=\Gamma(\nu-\delta_{x_1}-\cdots -\delta_{x_n},\cdot),
\quad \nu\in\bN.
\end{align*}
Therefore we obtain from \cite[Theorem 4.5]{LastPenrose18}
for each measurable $f\colon \BX^n\times\bG\to[0,\infty]$ that
\begin{align}\label{e:Mecke2}
\begin{split}
    \BE \int f(x_1,\ldots,x_n,\xi-\delta_{x_1}-\cdots -\delta_{x_n})\,\eta^{(n)}&(d(x_1,\ldots,x_n))\\
=& \BE \int f(x_1,\ldots,x_n,\xi)\,\lambda^n(d(x_1,\ldots,x_n)).
\end{split}
\end{align}
Given $v\in\BX$ we also have
\begin{align}\label{e:Mecke2v}
\BE \int f(x,\xi^v-\delta_x)\,\eta(dx)= \BE \int f(x,\xi^v)\,\lambda(dx).
\end{align}
This follows similarly as \eqref{e:Meckev}. Indeed, given $\eta$ and $x\in\eta$
the random graph $h(\xi,v,U)-\delta_x$ has the same distribution as $h(\xi-\delta_x,v,U)$,
provided that $v\ne x$.

Another quick consequence of the multivariate Mecke equation is the
following {\em deletion tolerance} of $\xi$.
Removing a finite number of points from $\eta$ results in a random graph
whose distribution is absolutely continuous with respect to the distribution of $\xi$. 
Deletion tolerance of point processes was studied in \cite{HolSoo13}.

\begin{proposition}\label{p:deltolerant} Let $A\subset \bG$ be a measurable set
such that $\BP(\xi\in A)=0$.
Let $\eta_0$ be a point process such that $\BP(\eta_0(\BX)<\infty)=1$ and $\BP(\eta_0\le\eta)=1$. Then
$\BP(\xi[\eta-\eta_0]\in A)=0$. 
\end{proposition}
\begin{proof} 
Let $n\in\N$. 
By the Mecke equation \eqref{e:Mecke2} we have
\begin{align*}
0=\int \BP(\xi\in A)\,\lambda^n(d(x_1,\ldots,x_n))=\BE \int \I\{\xi-\delta_{x_1}-\cdots-\delta_{x_n} \in A\}
\,\eta^{(n)}(d(x_1,\ldots,x_n)).
\end{align*}
Since $\BP(\eta_0(\BX)<\infty)=1$ it follows that
\begin{align*}
&\BP(\xi[\eta-\eta_0]\in A)\\
&=\sum^\infty_{n=0}\frac{1}{n!}\BE \I\{\eta_0(\BX)=n\}\int\I\{\xi-\delta_{x_1}-\cdots-\delta_{x_n} \in A\} \,\eta^{(n)}_0(d(x_1,\ldots,x_n))=0.
\end{align*}
The result follows.
\end{proof}

We define $\bar\varphi:=1-\varphi$ and for $x\in\BX$, $\nu\in\bN$
\begin{align}\label{e:barvarphi}
\bar\varphi(\nu,x):=\prod_{y\in\nu}\bar\varphi(x,y), \quad \varphi(\nu,x):=1-\bar\varphi(\nu,x), \
\quad \varphi_\lambda(\nu):=\int \varphi(\nu,x)\,\lambda(dx).
\end{align}
We recall our general convention $\varphi(\mu,x):=\varphi(V(\mu),x)$ 
and $\varphi_\lambda(\mu):=\varphi_\lambda(V(\mu))$ for $\mu\in\bG$.
Throughout we often abbreviate
$C^v:=C^v(\xi^v)$, 
$C^v_n:=C^v_n(\xi^v)$ and $C^v_{\le n}:=C^v_{\le n}(\xi^v)$. Moreover we write
$C^{v!}:=C^v-\delta_v$.  

We shall need the following consequence of
\eqref{e:Mecke}.

\begin{lemma}\label{l:312} Let $v\in\BX$ and $h\colon\bG\to[0,\infty)$ be measurable.
Then
\begin{align}
\BE \int h(\xi^v-\delta_x)\, C^{v!}(dx)=\BE h(\xi^v)\varphi_\lambda(C^v).
\end{align}
\end{lemma}
\begin{proof} Let $I$ denote the left-hand side of the asserted formula.
Then
\begin{align*}
I&=\BE \int h(\xi^v-\delta_x)\I\{x\in C^v(\xi^v)\}\, \eta(dx)
=\int \BE h(\xi^{v,x}-\delta_x)\I\{x\in C^v(\xi^{v,x})\}\, \lambda(dx),
\end{align*}
where we have used the Mecke equation \eqref{e:Meckev} to obtain the second identity.
By definition we have that
$\xi^{v,x}-\delta_x=\xi^v$ for each $x\in\BX$. Hence we obtain that
\begin{align*}
I&=\int \BE h(\xi^{v})\I\{x\in C^v(\xi^{v,x})\}\, \lambda(dx)
=\int \BE h(\xi^{v})\BP(x\in C^v(\xi^{v,x})\mid \xi^v)\, \lambda(dx).
\end{align*}
By definition of $\xi^{v,x}$ we have $\BP(x\in C^v(\xi^{v,x})\mid \xi^v)=\varphi(x,C^v)$,
concluding the proof.
\end{proof}

Next we turn to the {\em Margulis--Russo formula}. Let $\lambda_1$ and $\lambda_2$ be two
measures on $\BX$, where $\lambda_1$ is  locally finite and $\lambda_2$ is finite.
Given $t\ge 0$ we consider a RCM driven by a Poisson process $\eta$ with
intensity measure $\lambda_1+t\lambda_2$. The associated
expectation operator is denoted by $\BE_t$. Let
$f\colon \bG\to [-\infty,\infty]$ be a measurable function and
assume that $\BE_{t_0}|f(\xi)|<\infty$ for some $t_0>0$.  From
\cite[Exercise 3.8]{LastPenrose18} and \eqref{e:kernel} we then obtain
that $\BE_t|f(\xi)|<\infty$ for all $t\le t_0$.  We assert that
\begin{align}\label{MargulisRusso}
\frac{d}{d t} \BE_t f(\xi) = \int \BE_t[f(\xi^x) - f(\xi)] \,\lambda_2(dx),\quad t\in [0,t_0). 
\end{align}
Using the kernel $\Gamma$, this can be seen as follows.  From
\cite[Theorem 19.3]{LastPenrose18} we obtain that
\begin{align*}
\frac{d}{d t} \BE_tf(\xi) 
= \int \BE_t[\tilde{f}(\eta^x) - \tilde{f}(\eta)] \,\lambda_2(dx),\quad t\in[0,t_0), 
\end{align*}
where $\tilde{f}(\nu):=\int f(\mu)\,\Gamma(\nu,d\mu)$, $\nu\in\bN$. 
Note that $\tilde f(\xi)$ is $\BP_t$-a.s.\ well-defined.

Take $t\in[0,t_0)$. Theorem 19.3 in \cite{LastPenrose18} shows that
$\int \BE_t[|\tilde{f}(\eta^x) - \tilde{f}(\eta)|] \,\lambda_2(dx)<\infty$.
Furthermore we have
\begin{align*}
\int \BE_t[|\tilde{f}(\eta)|]\,\,\lambda_2(dx)=\lambda_2(\BX)\BE_t[|\tilde{f}(\eta)|]
\le \lambda_2(\BX)\BE_t[|f(\xi)|]<\infty,
\end{align*}
where we have used the triangle inequality and \eqref{e:kernel}. \\ 
Therefore we also have $\int \BE_t[|\tilde{f}(\eta^x)|]\,\lambda_2(dx)<\infty$.
It follows that
\begin{align*}
\frac{d}{d t} \BE_tf(\xi) 
= \int \big(\BE_t\tilde{f}(\eta^x) - \BE_t\tilde{f}(\eta)\big) \,\lambda_2(dx) 
=\int \big(\BE_tf(\xi^x) - \BE_tf(\xi)\big) \,\lambda_2(dx),
\end{align*}
where we have used 
\eqref{c1}. 
Since the above right-hand side is finite we have
$|\BE_tf(\xi^x)|<\infty$ and hence also $\BE_t|f(\xi^x)|<\infty$ 
for $\lambda_2$-a.e.\ $x$. This implies \eqref{MargulisRusso}.

\section{The stationary marked RCM}\label{s:stationaryRCM}

In this section we introduce an important special case of the general RCM. The setting is that of 
\cite{CaicedoDickson24, DicksonHeydenreich22}.
Special cases were studied in \cite{DepW15, GLM21, GHMM22}.

Let $\BM$ be a complete separable metric space equipped with a
probability measure $\BQ$. This is our {\em mark space}, while $\BQ$
is said to be the {\em mark distribution}.  In this section we
consider the space $\BX=\R^d\times\BM$ equipped with the product of
the Borel $\sigma$-field $\cB(\R^d)$ on $\R^d$ and the Borel
$\sigma$-field on $\BM$. We assume that $\lambda=t\lambda_d\otimes\BQ$,
where $t\in \R_+$ and $\lambda_d$ denotes Lebesgue measure on $\R^d$.  If
$(x,p)\in\BX$ then we call $x$ location of $(x,p)$ and $p$ the mark of
$x$. Instead of $\bN$ we consider the (smaller set) $\bN(\BX)$ of all
counting measures $\chi$ on $\BX$ such that $\chi(\cdot\times\BM)$ is
locally finite (w.r.t.\ the Euclidean metric) and simple.  The
symmetric connection function $\varphi\colon (\R^d\times\BM)^2\to [0,1]$
is assumed to satisfy\begin{align}\label{e:4.1} \varphi((x,p),(y,q))=
  \varphi((0,p),(y-x,q)).
\end{align}
This allows us to write $\varphi(x,p,q):=\varphi((0,p),(x,q))$, where
$0$ denotes the origin in $\R^d$.  
We also assume that 
\begin{align}\label{e:intmarked}
d_\varphi:=\iint \varphi(x,p,q)\,dx\,\BQ^2(d(p,q))<\infty,
\end{align}
referring to Remark \ref{r:int} for some comments.
Let $t>0$ and let $\eta$ be a Poisson process on $\BX$
with intensity measure $t\lambda$. 
We can and will assume that $\eta$ is a random element of $\bN(\BX)$.
We consider a RCM $\xi$ based on $\eta$ and connection function $\varphi$.

The RCM $\xi$ is {\em stationary} in the sense
that $T_x\xi\overset{d}{=}\xi$, $x\in\R^d$, where
for $\mu\in\bG$, the measure $T_x\mu$ is (as usual) defined by 
\begin{align*}
T_x\mu:=\int \I\{(y-x,q,\nu)\in\cdot\}\,\mu(d(y,q,\nu)).
\end{align*}
To see this, it is convenient to define $\xi$ in a slightly different way,
without changing its distribution. As at \eqref{e:points} we can write 
\begin{align}\label{e:Qm}
\eta=\sum^\infty_{m=1}\delta_{(X_m,Q_m)},
\end{align}
where $X_1,X,\ldots$ are pairwise distinct random elements
of $\R^d$ and $Q_1,Q_2,\ldots$ are random elements of $\BM$.
Let $Z'_{m,n}$, $m,n\in\N$, be independent random variables uniformly distributed
on $[0,1]$ and set $Z'_m:=(Z'_{m,n})_{n\in\N}$, $m\in\N$.
By the marking theorem (see \cite[Theorem 5.6]{LastPenrose18}), 
\begin{align}\label{e:eta*}
\eta^*:=\sum^\infty_{m=1}\delta_{(X_m,Q_m,Z'_m)}  
\end{align}
is again a Poisson process. 
We then connect $(X_m,Q_m)$ with $(X_n,Q_n)$ if $X_m$ is lexicographically
smaller than $X_n$ and $Z'_{m,\tau}\le  \varphi (X_n-X_m,Q_m,Q_n)$, where the $\N$-valued random variable
$\tau$ is determined by the fact that $X_n$ is the $\tau$-th nearest neighbour
of $X_m$ in the set $\{X_k:k\ne m\}$, where we can use the lexicographic order
to break ties. Then we have $\xi=F(\eta^*)$ for a well-defined 
measurable mapping $F$. Since the nearest neighbour relation
is translation invariant it follows from \eqref{e:4.1} that $F$ can be assumed to satisfy
$T_x\xi=F(T_x\eta^*)$ for each $x\in\R^d$. Since $T_x\eta^*\overset{d}{=}\eta^*$ it follows
that $\xi$ is stationary. The same argument combined with
\cite[Exercise 10.1]{LastPenrose18} shows that $\xi$ is {\em ergodic}, i.e.\
we have $\BP(\xi\in A)\in\{0,1\}$ for each translation invariant measurable
$A\subset \bG$. If $\BM$ contains only one element, we identify $\BX$ with $\R^d$.
In this case $\xi$ is said to be a {\em stationary} RCM.

The following consequence of the Mecke equation will be often used to treat cluster expectations.

\begin{lemma}\label{l:915} Let $B\in\cB(\R^d)$ and $f\colon\N\to \R_+$. Then
\begin{align}\label{e:Palm}
 \BE_t \int \I\{x\in B\}f(|C^{(x,p)}(\xi)|)\,\eta(d(x,p))
=t \lambda_d(B) \BE_t \int f(|C^{(0,p)}|)\,\BQ(dp).
\end{align}
\end{lemma}
\begin{proof}
By the Mecke equation \eqref{e:Mecke} the left-hand side of \eqref{e:Palm} equals
\begin{align*}
t\, \BE_t& \iint \I\{x\in B\} f(|C^{(x,p)}(\xi^{(x,p)})|)\,dx\,\BQ(dp) \\
=t\, \BE_t& \iint \I\{x\in B\} f(|C^{(0,p)}(T_x\xi^{(x,p)})|)\,dx\,\BQ(dp),
\end{align*}
where we have used that $|C^{(x,p)}(\mu)|=|C^{(0,p)}(T_x\mu)|$ for all $\mu\in\bG$.
It follows from stationarity of $\xi$ and definition of $\xi^{(x,p)}$, that $T_x\xi^{(x,p)}\overset{d}{=}\xi^{(0,p)}$ for $\lambda_d\otimes\BQ$-a.e.\ 
$(x,p)\in \R^d\times\BM$. Therefore the result follows.
\end{proof}

Let $Q_0$ be a random element of $\BM$ with distribution $\BQ$ which is independent
of $\eta^*$ given by \eqref{e:eta*}. In accordance with Palm theory
we refer to $C^{(0,Q_0)}(\xi^{(0,Q_0)})$ as cluster of the {\em typical vertex} (of $\xi$). 

\begin{remark}\label{r:int}\rm Let $p\in\BM$. Then the
degree $D_p$ of $(0,p)$ (the origin marked with $p$) in $\xi^{(0,p)}$ has a Poisson distribution
with parameter $t\int\varphi(x,p,q)\,dx\,\BQ(dq)$. 
Our integrability assumption \eqref{e:intmarked} means
that $\int \BE D_p\,\BQ(dp)<\infty$. This means that
the expected degree of the 
typical vertex is finite. Hence \eqref{e:intmarked} excludes Pareto type degree distributions 
but is still much weaker than the integrability assumption made in
\cite{CaicedoDickson24}.  
\end{remark}

The function 
\begin{align}\label{e:markedfree}
\kappa(t):=\int \BE_t |C^{(0,p)}|^{-1}\,\BQ(dp)=\BE_t |C^{(0,Q_0)}|^{-1},\quad t\in\R_+,
\end{align}
plays a crucial role in Section \ref{s:tolRCM}.
To interpret it, we introduce a point process $\eta_c\le \eta(\cdot\times\BM)$ 
modeling finite clusters as follows. 
Let $(x,p)\in\eta$. Then $x\in\eta_c$ if $|C^{(x,p)}(\xi)|<\infty$
and $x$ is the lexicographically smallest spatial coordinate of the points in
$C^{(x,p)}(\xi)$. Since $\xi$ is stationary, it is easy to see that $\eta_c$
is a stationary point process. The following result shows that $t\kappa(t)$ is the {\em intensity}
of $\eta_c$, that is the density of finite clusters.  With a slight
abuse of language we refer to $\kappa(t)$ as {\em cluster density}. In the unmarked
case this function is also called {\em free energy}; see \cite{Aizenman87,BeGrimLoff98,DurrettNguyen85}.

\begin{lemma}\label{l:kappa} For each $t\in\R_+$ we have that
$t \kappa(t)=\BE_t \eta_c([0,1]^d)$.
\end{lemma}
\begin{proof} The result follows from \cite[Proposition 3.1]{LastZiesche17}
upon taking there $\eta$ as the projection of the point process $\{(x,p)\in\eta: |C^{(x,p)}(\xi)|<\infty\}$
onto $\R^d$ and $\xi:=\eta_c$.
A direct proof can start with
\begin{align}\label{e:intensity}
t  \kappa(t)=\BE_t \int \I\{x\in [0,1]^d\}|C^{(x,p)}(\xi)|^{-1}\,\eta(d(x,p)),
\end{align}
a consequence of \eqref{e:Palm}. The right-hand side can be written  as
\begin{align*}
\BE_t \iint \I\{x\in [0,1]^d\}|C^{(x,p)}(\xi)|^{-1}\I\{\tau(x,p)=y\}\,\eta(d(x,p))\,\eta_c(dy),
\end{align*}
where $\tau(x,p)$ is the lexicographic minimum of the spatial coordinates of
$C^{(x,p)}(\xi)$. The key step is then the application of the refined Campbell theorem for
$\eta_c$.  
\end{proof}

The cluster density can also be obtained as an ergodic limit:

\begin{proposition}\label{p:ergodiclimit} 
Let $B_n\in\cB(\R^d)$, $n\in\N$, be an increasing sequence of compact
convex sets whose inradius diverges to $\infty$. Then

\begin{align*}
\lim_{n\to\infty}(\lambda_d(B_n))^{-1}\int \I\{x\in B_n\}|C^{(x,p)}(\xi)|^{-1}\,\eta(d(x,p))=t\kappa(t),\quad \BP_t\text{-a.s.}
\end{align*}

\end{proposition}
\begin{proof}
For each $\mu\in\bG$ 
\begin{align*}
M_\mu:=\int \I\{x\in \cdot\}|C^{(x,p)}(\mu)|^{-1}\,\mu(d(x,p))
\end{align*}
is a locally finite measure on $\R^d$. 
For $x,y\in\R^d$ and $\mu\in\bG$ we have 
 $C^{(x,p)}(T_y\mu)=T_yC^{(x+y,p)}(\mu)$. Therefore, we obtain for $B\in\cB(\R^d)$ and $y\in\R^d$ 
\begin{align*}
M_{T_y\mu}(B)&:=\int \I\{x\in B\}|C^{(x+y,p)}(\mu)|^{-1}\,T_y\mu(d(x,p))\\
&=\int \I\{x-y\in B\}|C^{(x,p)}(\mu)|^{-1}\,\mu(d(x,p)).
\end{align*}
This means that $M_{T_y\mu}=T_yM_{\mu}$. Therefore $M_\xi$ is a stationary
and ergodic random measure. By \eqref{e:intensity} it has intensity $t\kappa(t)$.
Hence the result follows from \cite[Satz 3]{NguyenZessin76}; 
see also \cite[Theorem 30.10]{Kallenberg}.
 \end{proof}

We continue with a basic fact from percolation theory. Define
\begin{align}\label{e:typ}
\theta(t):=\BP_t\big(\big|C^{(0,Q_0)}(\xi^{(0,Q_0)})\big|=\infty\big)
  =\int\BP_t\big(|C^{(0,p)}|=\infty\big)\,\BQ(dp),\quad t\ge 0,
\end{align}
as the probability that the cluster of a typical vertex has infinite size.
Let $C_\infty$ denote the set of all $\mu\in\bG$ such that  $\mu$ has an infinite cluster.

\begin{proposition}\label{p:ergodic} Let $t>0$. Then  $\theta(t)>0$ if and only if $\BP_t(\xi\in C_\infty)=1$.
\end{proposition}
\begin{proof} Let  $B\in\cB(\R^d)$ be a Borel set with $\lambda_d(B)\in(0,\infty)$.
By \eqref{e:Palm},
\begin{align*}
\theta(t)=(t\lambda_d(B))^{-1}\BE_t \int \I\{x\in B\}\I\{|C^{(x,p)}(\xi)|=\infty\}\,\eta(d(x,p)).
\end{align*}

Hence,  if $\theta(t)>0$, then the probability that there is some $(x,p)\in\eta$ with
$|C^{(x,p)}(\xi)|=\infty$ must be positive. Since $\xi$ is ergodic and $C_\infty$ is
translation invariant, we obtain $\BP_t(\xi\in C_\infty)=1$.
If, on the other hand, $\theta(t)=0$, then the probability that
$|C^{(x,p)}(\xi)|=\infty$ for some $(x,p)\in\eta$ with $x\in B$ is zero.
Letting $B\uparrow\R^d$ we obtain that $\BP_t(\xi\in C_\infty)=0$.
\end{proof}

The {\em critical intensity} (percolation threshold) is defined by
\begin{align}
t_c:=\inf\{t>0:\theta(t)>0\}.
\end{align}
If $t<t_c$ then $\BP_t(\xi\in C_\infty)=0$ and if $t>t_c$ then $\BP_t(\xi\in C_\infty)=1$.
Under a natural
irreducibility assumption 
our Theorem \ref{t:markedunique} will show that $\xi$ can have at most
one infinite cluster.

We finish this section with some examples.

\begin{example}\label{ex:classical}\rm In the unmarked case the
connection function $\varphi$ is just a function on $\R^d$. 
Under the minimal assumption 
$d_\varphi\in(0,\infty)$
it was shown in \cite{Pen91} that $t_c\in(0,\infty)$. 
\end{example}

\begin{example}\label{ex:Gilbert}\rm Assume that $\BM=\R_+$ and
$\varphi(x,p,q)=\I\{\|x\|\le p+q\}$, where $\|x\|$ denotes the Euclidean norm of $x\in\R^d$.
The RCM $\xi$ is then said to be the {\em Gilbert graph} with
radius distribution $\BQ$; see e.g.\ \cite[Chapter 16]{LastPenrose18} for more detail.
The integrability assumption \eqref{e:intmarked} is then equivalent with
$\int r^d\,\BQ(dr)<\infty$, which is the minimal assumption for having
a reasonable model.
Under the additional assumption $\BQ\{0\}<1$ it was proved in \cite{Gou09, Ha85} that $t_c\in(0,\infty)$.
\end{example}

\begin{example}\label{ex:Boolean}\rm Suppose that $\BM$ equals the space $\mathcal{C}^d$
of all non-empty compact subsets of $\R^d$, equipped with the Hausdorff metric
cf.\ \cite{LastPenrose18, SW08}. 
Let $V\colon \mathcal{C}^d\cup\{\emptyset\}\to[0,\infty]$ be measurable and translation invariant
with $V(\emptyset)=0$. For instance, $V$ could be the volume or, if $\BQ$
is concentrated on the convex bodies, a linear combination
of the {\em intrinsic volumes}; see \cite{SW08}.
Assume that the connection function is given by
\begin{align*}
\varphi((x,K),(y,L))=1-e^{-V((K+x)\cap (L+y))}, \quad (x,K),(y,L)\in\R^d\times\mathcal{C}^d.
\end{align*}
Then \eqref{e:4.1} follows from translation invariance of $V$. 
The case of the Gilbert graph arises if $\BQ$ is concentrated on balls centered
at the origin. A sufficient condition for \eqref{e:intmarked} is
\begin{align*}
\int D(K)^d\,\BQ(dK)<\infty,
\end{align*}
where $D(K)$ is the radius of the smallest ball centered at the origin
and containing $K$. This easily follows from
\begin{align*}
\varphi(x,K,L)
\le \I\{K\cap (L+x)\ne\emptyset\} \le \I\{\|x\|\le D(K)+D(L)\}. 
\end{align*}
The random closed set $\bigcup_{(x,K)\in\eta}K+x$ is known as the {\em Boolean model}
and a fundamental model of stochastic geometry (see \cite{LastPenrose18, SW08})
and continuum percolation (see \cite{MeesterRoy}).
This model corresponds to the choice $V(K)=\infty\cdot\I\{K\ne \emptyset\}$.
In that case and under some additional assumptions on $\BQ$ 
it was proved in \cite{Ha85} that $t_c\in(0,\infty)$.
The present much more general model
is taken from \cite{BetschLast21} and is partially motivated by
statistical physics.
\end{example}	
      
\begin{example}\label{ex:weighted}\rm Assume that $\BM=(0,1)$ 
equipped with Lebesgue measure $\BQ$. Assume that
\begin{align*}
\varphi((x,p),(y,q))=\rho(g(p,q)\|x-y\|^d),
\end{align*}
for a decreasing function $\rho\colon [0,\infty)\to [0,1]$ and a 
function $g\colon(0,1)\times(0,1)\to [0,\infty)$ which is increasing in both arguments.
We assume that $m_\rho:=\int \rho(\|x\|^d)\,dx$ is positive and finite. 
This model was studied in \cite{GHMM22}
under the name {\em weight-dependent random connection model}.
A simple calculation shows that
\begin{align*}
d_\varphi=m_\rho\iint g(p,q)^{-1}\,dp\,dq.
\end{align*}
To ensure \eqref{e:intmarked} we have to assume that $g^{-1}$ is integrable. This is the case
in all examples studied in \cite{GHMM22}, where it is also asserted that $t_c<\infty$.
Sufficient conditions for $t_c\in(0,\infty)$ can also be found in  \cite{CaicedoDickson24,DepW15}. 
\end{example}

\section{Irreducibility}\label{sec:irred}

In this section, we first consider a general RCM $\xi$ based on a Poisson process $\eta$ on $\BX$ with diffuse intensity measure $t\lambda$. 
We fix the intensity parameter $t>0$ and therefore drop the lower index
$t$ in $\BP_t$. To simplify the notation, we take $t=1$.
We say that $\xi$ is {\em irreducible} if 
\begin{align}\label{e:irred}
\BP(\text{$x_1 \leftrightarrow  x_2$ in $\xi^{x_1,x_2}$})>0,\quad \lambda^2\text{-a.e. $(x_1,x_2)\in\BX^2$}.
\end{align}

Given $k\in\N$ and random elements $Y_1,\ldots,Y_k$ of $\BX$ we let
$\Xi[Y_1,\dots,Y_k]$ be a RCM
based on the point process $\delta_{Y_1}+\cdots + \delta_{Y_k}$.
Of course we can allow here some of the $Y_1,\ldots,Y_k$ to be deterministic.
Further we define for each $n\in\N$ a measurable function
$\varphi^{(n)}\colon \BX^n\to[0,\infty)$ recursively by
$\varphi^{(1)}:=\varphi$ and
\begin{align*}
\varphi^{(n+1)}(x_1,x_2):=\int \varphi^{(n)}(x_1,z)\varphi(z,x_2)\,\lambda(dz),\quad x_1,x_2\in\BX,\,n\in\N. 
\end{align*}
These functions are symmetric. It follows straight from the Mecke equation \eqref{e:Mecke}
that
\begin{align}\label{e:paths}
\varphi^{(n)}(x_1,x_2)
=\BE \int \prod_{i=1}^{n}\I\{\text{$y_{i-1}\sim y_i$ in $\xi^{x_1,x_2}$}\}\, \eta^{(n-1)}(d(y_1,\dots,y_{n-1})),
\end{align}
where $y_0:=x_1$ and $y_{n}:=x_2$. This is the expected number of paths of length
$n$ from $x_1$ to $x_2$ in  $\xi^{x_1,x_2}$.

\begin{proposition}\label{p:irred} The following six  statements are equivalent: 
\begin{itemize}
    \item[{\rm (i)}] $\xi$ is irreducible.
    \item[{\rm (ii)}] There exist for $\lambda^2$-a.e.\ $(x_1,x_2)\in\BX^2$ a set $B\in\cX$ with
$\lambda(B)\in(0,\infty)$, an  $n\in\N_0$ and independent random
variables $Y_1,\ldots,Y_n$ with distribution $\lambda_B/\lambda(B)$ such
that
\begin{align}\label{e:pirr1}
\BP(\text{$x_1 \leftrightarrow  x_2$ in $\Xi[x_1,x_2,Y_1,\ldots,Y_n]$})>0.
\end{align}
\item[{\rm (iii)}] There exists for $\lambda^2$-a.e.\ $(x_1,x_2)\in\BX^2$ an  $n\in\N_0$ such
that
\begin{align}\label{e:pirr2}
\int \BP(\text{$x_1 \leftrightarrow  x_2$ in $\Xi[x_1,x_2,y_1,\ldots,y_n]$}) \lambda^n(d(y_1,\dots,y_n))>0.
\end{align}
\item[{\rm (iv)}]  For $\lambda^2$-a.e.\ $(x_1,x_2)\in\BX^2$ it holds that
\begin{align}\label{(iv)}
\sup_{n\ge 1}\varphi^{(n)}(x_1,x_2)>0.
\end{align}
\item[{\rm (v)}] For $\lambda^2$-a.e.\ $(x_1,x_2)\in\BX^2$  and for all $k\in\N$ we have
\begin{align}\label{(v)}
\sup_{n\ge k}\varphi^{(n)}(x_1,x_2)>0.
\end{align} 
\item[{\rm (vi)}] There exist for $\lambda^2$-a.e.\ $(x_1,x_2)\in\BX^2$ a set $B\in\cX$ with
$\lambda(B)\in(0,\infty)$, an  $n\in\N$ and independent random
variables $Y_1,\ldots,Y_n$ with distribution $\lambda_B/\lambda(B)$ such
that
\begin{align}\label{(vi)}
\BP(\text{$x_1 \leftrightarrow  x_2$ in $\Xi'[x_1,x_2,Y_1,\ldots,Y_n]$})>0,
\end{align}
where $\Xi'[x_1,x_2,Y_1,\ldots,Y_n]$ is the graph obtained from $\Xi[x_1,x_2,Y_1,\ldots,Y_n]$
by removing the edge between $x_1$ and $x_2$ (if there is one).
\end{itemize}
\end{proposition}
\begin{proof} Assume that (i) holds, and let $x_1,x_2\in\BX$ be such that $\BP(\text{$x_1 \leftrightarrow  x_2$ in $\xi^{x_1,x_2}$})>0$.
Let $(B_m)$ be a sequence of measurable sets of finite $\lambda$-measure
increasing towards $\BX$. By monotone convergence
there exists $m\in\N$ such that $x_1,x_2\in B_m$ and
\begin{align}
\BP(\text{$x_1 \leftrightarrow  x_2$ in $\xi^{x_1,x_2}[B_m]$})>0.
\end{align}
Let $B:=B_m$. Note that $\xi^{x_1,x_2}[B]$ is a RCM based on $\delta_{x_1}+\delta_{x_2}+\eta_B$
and $\eta_B\overset{d}{=}\sum^{\eta(B)}_{k=0}Y_k$, where $Y_1,Y_2,\ldots$ are independent with distribution
$\lambda_B/\lambda(B)$, independent of $\eta_B$. Splitting the event
$\{\text{$x_1 \leftrightarrow  x_2$ in $\xi^{x_1,x_2}[B]$}\}$ according to the value
of $\eta(B)$ yields (ii). 

Assume that (ii) holds. Then we have for  $\lambda^2$-a.e.\ $(x_1,x_2)$ that
\begin{align*}
\BP&(\text{$x_1 \leftrightarrow  x_2$ in $\Xi[x_1,x_2,Y_1,\ldots,Y_n]$})\\
=&\lambda(B)^{-n} \int \BP(\text{$x_1 \leftrightarrow  x_2$ in $\Xi[x_1,x_2,y_1,\ldots,y_n]$})\, 
\lambda_B^n(d(y_1,\dots,y_n))>0,
\end{align*}
which implies (iii).

Assume 
that $x_1,x_2\in\BX$ satisfy  \eqref{e:pirr2}.
If $x_1 \leftrightarrow  x_2$ in $\Xi[x_1,x_2,y_1,\ldots,y_n]$ then there exist
$k\in\{0,\ldots,n\}$ and pairwise distinct $i_1,\ldots,i_k\in \{1,\ldots,n\}$ with
$x_1\sim y_{i_1}\sim\cdots\sim y_{i_k}\sim x_2$ in $\Xi[x_1,x_2,y_1,\ldots,y_n]$.
Therefore and by the symmetry of $\lambda^n$ 
\begin{align*}
\sum^n_{k=0}\int \prod^{k+1}_{i=1} \varphi(y_{i-1},y_i)\,\lambda^k(d(y_1,\dots,y_k))>0, 
\end{align*}
where $y_0:=x_1$ and $y_{k+1}:=x_2$. Hence (iv) follows.

Assume that (iv) holds, then for $\lambda^2$-a.e.\ $(x_1,x_2)\in\BX^2$ 
\begin{align*}
0<\sum_{m,n=1}^\infty\varphi^{(m+n)}(x_1,x_2)
  = & \sum_{m,n=1}^\infty \int \varphi^{(m)}(x_1,z)\varphi^{(n)}(z,x_2)\,\lambda(dz) \\
  =& \int \left(\sum_{m=1}^\infty \varphi^{(m)}(x_1,z) \right) \left(\sum_{n=1}^\infty \varphi^{(n)}(z,x_2)\right)\,\lambda(dz).
\end{align*}  
Therefore $\sup\limits_{n\ge 2}\varphi^{(n)}(x_1,x_2)>0$ for $\lambda^2$-a.e.\ $(x_1,x_2)\in\BX^2$. 
To obtain \eqref{(v)} for general $k\in\N$ one has to start with a $k$-fold summation
instead of a double summation.

Assume now that $x_1,x_2\in\BX$ satisfy  \eqref{(v)}, then there exists $n\ge 2$ such that
$\varphi^{(n)}(x_1,x_2)>0$. Therefore there exists $B\in\cX$ of finite $\lambda$-measure
with
\begin{align*}
\int \prod^{n}_{i=1} \varphi(y_{i-1},y_i)\,\lambda^{n-1}_B(d(y_1,\dots,y_{n-1}))>0, 
\end{align*}
where $y_0:=x_1$ and $y_{n}:=x_2$. This implies (vi). 

Finally, note that for any $n\in\N$ 
\begin{align*}
\BP(\text{$x_1 \leftrightarrow  x_2$ in $\Xi'[x_1,x_2,Y_1,\ldots,Y_{n}]$})
&\le \BP(\text{$x_1 \leftrightarrow  x_2$ in $\Xi[x_1,x_2,Y_1,\ldots,Y_{n}]$}) \\ 
&=\BP(\text{$x_1 \leftrightarrow  x_2$ in $\xi^{x_1,x_2}[B]$}\mid \eta(B)=n).
\end{align*}
Therefore (vi) implies (i).  
\end{proof}

\begin{remark}\rm Proposition \ref{p:irred} shows that irreducibility does not depend
on the intensity parameter $t$ as long as it is positive.  
\end{remark}

\begin{remark}\label{r:non-diffuseirred}\rm Consider the setting of Remark \ref{r:non-diffuse},
where $\lambda$ is not assumed to be diffuse. In the diffuse case, the value of the connection
function on the diagonal of $\BX^2$ is irrelevant for the definition of the RCM. But now we assume that $\varphi(x,x)=1$ for
all $x\in\BX$. For each $n\in\N$ we can define a measurable function
$\hat\varphi^{(n)}\colon \hat\BX^n\to[0,\infty)$ as before.
Now it can easily be checked that
\begin{align*}
\hat\varphi^{(n)}((x,u),(y,v))=\varphi^{(n)}(x,y),\quad (x,u),(y,v)\in\hat\BX,
\end{align*}
where $\varphi^{(n)}$ is defined as before, non-withstanding the fact that $\lambda$ might not be diffuse.
By Proposition \ref{p:irred} we therefore have that the random connection model $\hat\xi$
is irreducible, if and only if \eqref{(iv)} holds. 
\end{remark}

\begin{example}\label{e:generalBoolean}\rm
Let $\BY$ be a locally compact separable Hausdorff space  
and assume that $\BX$ is the class of closed subsets of $\BY$
equipped with the Fell topology; see e.g.\ \cite{SW08}.
Let $\nu$ be a locally finite measure on $\BY$, $\BQ$ be a probability measure
on $[0,\infty)$, such that
$\int \nu(B(x,r))\,\BQ(dr)<\infty$ for all $x\in\BX$ and all $r\ge 0$,
where $B(x,r)$ denotes the closed ball with centre $x$ and radius $r$.
Assume that $\lambda$ is given by
\begin{align}
\lambda=\iint\I\{B(x,r)\in\cdot\}\,\nu(dx)\,\BQ(dr). 
\end{align}
It is easy to see, that $\lambda$ is locally finite. Similarly as in
Example \ref{ex:Boolean} we take a measurable function
$V\colon \BX\to[0,\infty]$ and assume that
$\varphi(K,L)=1-e^{-V(K\cap L)}$, $K,L\in\BX$.
Assume also that the function $V$ is increasing w.r.t.\ set inclusion and that
$V(K)>0$ if $K\ne \emptyset$. If $\BQ([\varepsilon,\infty))>0$ for some $\varepsilon>0$
and
\begin{align*}
\int \I\{B(x_1,\varepsilon)\cap B(x_2,\varepsilon)\ne\emptyset\}\,\nu^2(d(x_1,x_2))>0, 
\end{align*}
then $\varphi(K,L)>0$ for all $K,L\in\BX$ and the RCM $\xi$ is irreducible.
\end{example}

In the remainder of this section we consider the stationary marked RCM as discussed in 
Section \ref{s:stationaryRCM}.  
Recall that without loss of generality, we can take $t=1$.
It is easy to see that for all $n\in\N$ and all $(x,p),(y,q)\in\BX$
\begin{align}\label{e:trinvariance}
\varphi^{(n)}((x,p),(y,q))=\varphi^{(n)}((0,p),(y-x,q))=\varphi^{(n)}((0,q),(x-y,p))
\end{align}
We write $\varphi^{(n)}(x,p,q):=\varphi^{(n)}((0,p),(x,q))$ and note that
$\varphi^{(n)}(x,p,q)=\varphi^{(n)}(-x,q,p)$. In the unmarked case we can identify
$\BX$ with $\R^d$. In this case
$\varphi^{(n)}=\varphi^{\ast n}$ is the $n$-fold convolution of $\varphi$, where $\varphi$ is considered as function
on $\R^d$.

\begin{proposition}\label{p:irredunmarked} The stationary (unmarked) RCM is irreducible.
\end{proposition}

The proof of Proposition \ref{p:irredunmarked} is a quick consequence
of the first part of the following lemma.

\begin{lemma}\label{l:convolution} Assume that $f\colon\R^d\to [0,\infty)$
is a bounded measurable function with $0<\int f(y)\,dy<\infty$ and $f(y)=f(-y)$ for
all $y\in\R^d$. Let $R>0$. 
\begin{enumerate}
\item[{\rm (i)}]
There exist $n\in\N$ and $\varepsilon>0$ such that
$f^{\ast n}(x)\ge \varepsilon$ whenever $\|x\| \le R$. 
\item[{\rm (ii)}]
Let $g\colon\R^d\to \R_+$
be another bounded measurable function with $\int g(y)\,dy>0$
and let $x\in\R^d$. Then there exists $n\in\N$ such that
$(f^{\ast n}\ast g)(x)>0$.
\end{enumerate}
\end{lemma}
\begin{proof} (i) The convolution of an integrable and a bounded function is bounded and uniformly continuous;
see  \cite[Proposition 8.8]{Folland99}. It follows that
$f^{\ast 2}$ is bounded and uniformly continuous. Since $\int f^{\ast 2}(x)\,dx=\big(\int f(x)\,dx\big)^2>0$
there exist a ball $B'\subset \R^d$ with positive radius and $\varepsilon'>0$ such that
$f^{\ast 2}\ge \varepsilon'$ on $B'$. Since $f$ is symmetric,
$f^{\ast 2}$ is symmetric as well. Hence
we can find a ball $B$ with center $0$ and positive radius and some $\varepsilon>0$ such that
$f^{\ast 4}\ge \varepsilon$ on $B$. Finally we find $m\in\N$ and $\varepsilon>0$ such that 
$f^{\ast 4m}(x)\ge \varepsilon$ whenever $\|x\|\le R$.

(ii) By assumption $\int g(x)\I\{g(x)\ge \varepsilon_0\}\, dx>0$
for some $\varepsilon_0>0$. Set $C:=\{g\ge \varepsilon_0\}$.
Then we have for  each $n\in\N$ that
\begin{align*}
(f^{\ast n}\ast g)(x)\ge \varepsilon_0 \int f^{\ast n}(x-z)\I\{z\in C\}\,dz.
\end{align*}
Choose $R>0$ so large that
\begin{align*}
\int\I\{\|x-z\|\le R,z\in C\}\,dz>0.
\end{align*}
By the first part of the lemma we can find $n\in\N$ and 
$\varepsilon>0$ such that $f^{\ast n}(y)\ge \varepsilon$ whenever $\|y\|\le R$.
It follows that
\begin{align*}
(f^{\ast n}\ast g)(x) &\ge \varepsilon_0 \int f^{\ast n}(x-z)\I\{\|x-z\|\le R, z\in C\}\,dz\\
&\ge \varepsilon_0\varepsilon  \int \I\{\|x-z\|\le R, z\in C\}\,dz.
\end{align*}
By the choice of $R$ this is positive.
\end{proof}

\begin{proof}[Proof of Proposition \ref{p:irredunmarked}]
We can use Lemma \ref{l:convolution} (i) and condition \eqref{(iv)} from
Proposition \ref{p:irred} to conclude the proof.
Indeed, given $x_1,x_2\in\BX$ we find an $n\in\N$ such that
$\varphi^{(n)}(x_1,x_2)=\varphi^{\ast n}(x_2-x_1)>0$.
\end{proof}

It is natural to characterize irreducibility of the stationary marked RCM in
terms of the functions $d_\varphi^{(n)}\colon \BM^2\to[0,\infty]$, $n\in\N$, 
defined by
\begin{align*}
d_\varphi^{(n)}(p,q):=\int \varphi^{(n)}(x,p,q)\,dx,\quad p,q\in\BM.
\end{align*}
Similarly as at \eqref{e:paths} we see that $\int d^{(n)}_\varphi(p,q)\I\{q\in A\}\,\BQ(dq)$ 
is the expected number of paths of length
$n$ from $(0,p)$ to some location with mark in a measurable set $A\subset \BM$.
From the symmetry property  of $\varphi^{(n)}$ we obtain that $d_\varphi^{(n)}$ is
symmetric. Furthermore,
\begin{align*}
d_\varphi^{(n)}(p,q)=\int \prod_{i=1}^n d_\varphi(q_{i-1},q_i)\, \BQ^{n-1} (d(q_1,\ldots,q_{n-1})),
\end{align*}
where $d_\varphi(\cdot,\cdot):=d_\varphi^{(1)}(\cdot,\cdot)$, $q_0:=p$, $q_n:=q$. Therefore
\begin{align}\label{e:Dconv}
d_\varphi^{(m+n)}(p,q)=\int d_\varphi^{(m)}(p,r)d_\varphi^{(n)}(r,q)\, \BQ (dr),\quad p,q\in\BM,\,m,n\in\N.
\end{align}
We may interpret $q\mapsto \sum_{n=1}^\infty d^{(n)}(p,q)$ as a mark occupation density of the RCM.

\begin{theorem}\label{l:77}
Let $\xi$ be a stationary marked RCM. Then $\xi$ is irreducible if and only if  
    \begin{align}\label{e:lirr}
 \sup_{n\ge 1} d_\varphi^{(n)}(p,q)>0,\quad \BQ^2\text{-a.e. $(p,q)\in\BM^2$}.
    \end{align}
  \end{theorem}
\begin{proof} Assume that $\xi$ is irreducible and for the sake of contradiction that
there exists some measurable set  $B\subset \BM^2$ satisfying  $\BQ^2 (B)>0$,
and  $d_\varphi^{(n)}(p,q)=0$ for all $(p,q)\in B$ and for each  $n\in \N$. 
But then $\varphi^{(n)}((0,p),(x,q))=0$ for all
  $(p,q)\in B$, $n\in \N$ and $\lambda_d$-a.e.\ $x\in \R^d$. This  
 contradicts  Proposition \ref{p:irred} (iv).  

 Now we assume that \eqref{e:lirr} holds. Then $d_\varphi>0$. Indeed,
 otherwise we would have $\varphi(x,p,q)=0$ and then also
$\varphi^{(n)}(x,p,q)=0$, $n\in\N$, for $\lambda$-a.e.\ $(x,p,q)\in\R^d\times\BM^2$. 
Therefore there exist measurable sets
 $A\subset \R^d$ and $E,F\subset \BM$ such that
 $\lambda_d(A)>0$, $\BQ(E)>0$, $\BQ(F)>0$, and
\begin{align}\label{e:inf_con_fun}
   \varepsilon_0:= \inf_{x\in A, p\in E, q\in F}\varphi(x,p,q) 
= \inf_{x\in -A, p\in E, q\in F}\varphi(x,q,p)>0.
\end{align}
Set $\varepsilon(x):=\varepsilon_0\I\{x\in A\}$. 
Then $\varphi(x,p,q)\ge \varepsilon(x)$ for each $(x,p,q)\in\R^d\times E\times F$.  
For all $x\in\R^d$ and all $r,s\in F$ we have
\begin{align*}
\varphi^{(2)}(x,r,s)&=\iint \varphi(z,r,q) \varphi(x-z,q,s)\,\BQ(dq)\,dz\\
&=\iint \varphi(-z,q,r) \varphi(x-z,q,s)\,\BQ(dq)\,dz\\
&\ge \BQ(E)\int \varepsilon(-z) \varepsilon(x-z)\,dz
=\BQ(E)f(x),
 \end{align*}
where $f$ is the convolution of $\varepsilon$ and $\varepsilon_0\I_{-A}$.
Note that $f$ is symmetric and has a positive and finite integral.
We further have
\begin{align*}
\varphi^{(4)}(x,r,s)
&\ge \iint\I\{q\in F\} \varphi^{(2)}(z,r,q) \varphi^{(2)}(x-z,q,s)\,\BQ(dq)\,dz
\ge\BQ(E)^2\BQ(F)f^{\ast 2}(x)
\end{align*}
and, inductively,
\begin{align}\label{e:4871}
\varphi^{(2n)}(x,r,s)\ge \BQ(E)^n\BQ(F)^{n-1}f^{\ast n}(x).
\end{align}

Our goal is to use Proposition \ref{p:irred} (iv). Let $p,q\in\BM$.
In view of this goal and assumption \eqref{e:lirr}  
we can assume that there exist $k,l\in\N$ such that
\begin{align*}
    \int_F d^{(k)}_\varphi(p,r) \,\BQ(dr) >0 \quad \text{and} \quad \int_F d^{(l)}_\varphi(r,q) \,\BQ(dr) >0. 
\end{align*}
Then we obtain by \eqref{e:4871} for each $x\in\R^d$
and each $n\in\N$ that
\begin{align*}
\varphi^{(k+2n+l)}(x,p,q)=&
\iint \varphi^{(k)}(z,p,r)\varphi^{(2n)}(w,r,s) \varphi^{(l)}(x-z-w,s,q)\,\BQ^2(d(r,s))\,d(z,w)\\
\ge \BQ(E)^n \BQ(F)^{n-1}&\iint_{F^2} \varphi^{(k)}(z,p,r) f^{\ast n}(w) \varphi^{(l)}(x-z-w,s,q)\,\BQ^2(d(r,s))\,d(z,w)\\
=\BQ(E)^n \BQ(F)^{n-1}& (g\ast f^{\ast n}\ast h)(x)=\BQ(E)^n \BQ(F)^{n-1} (f^{\ast n}\ast g \ast h)(x),
\end{align*}
where $g(z):=\int_F \varphi^{(k)}(z,p,r)\,\BQ(dr)$ and $h(z):=\int_F \varphi^{(l)}(z,r,q)\,\BQ(dr)$.
By the choice of $k$ and $l$ we have 
\begin{align*}
\int g\ast h(x)\,dx=\int g(x)\,dx\int h(x)\,dx>0.
\end{align*}
Therefore, given $x\in\R^d$, we obtain from Lemma \ref{l:convolution}
that $(f^{\ast n}\ast g \ast h)(x)$ is positive for some sufficiently large $n$.
Taking into the account translation invariance \eqref{e:trinvariance} we hence obtain
from Proposition \ref{p:irred} (iv) that $\xi$ is irreducible.
\end{proof}

\begin{remark}\label{r:irred}\rm Condition \eqref{e:lirr} can be expressed as follows.
For each $B\in\cX$ with $\BQ(B)>0$ we have
\begin{align}\label{e:lirr2}
\int_B\sum^\infty _{n=1}d_\varphi^{(n)}(p,q)>0,\quad \BQ\text{-a.e.\ $p$}.
\end{align}
Assume (for the sake of illustration) that $d(p):=\int d_\varphi(p,q)\,\BQ(dq)>0$  
for each $p\in\BM$ and let $(X_n)_{n\ge 0}$ be a Markov chain with transition kernel 
$p\mapsto d(p)^{-1}d_\varphi(p,q)\BQ(dq)$. Then
condition \eqref{e:lirr2} means that
\begin{align}\label{e:lirr3}
\BE\bigg[ \sum^\infty _{n=1}\I\{X_n\in B\}\bigg| X_0=p\bigg]>0,\quad \BQ\text{-a.e.\ $p$},
\end{align}
which is slightly weaker than $\BQ$-irreducibility of $(X_n)$, as studied in \cite{MeynTweedie}.
Without additional assumptions on $\varphi$, \eqref{e:lirr} seems to be both a natural
and a minimal assumption. 
If there exists $A\in\cX$ with $0<\BQ(A)<1$ and
$d_\varphi(p,q)=0$ for $(p,q)\in A\times A^c$, then $d^{(n)}_\varphi(p,q)=0$ for each $n\in\N$ and
all $(p,q)\in A\times A^c$, so that \eqref{e:lirr} fails.
\end{remark}

The next remark shows the relevance of irreducibility for the uniqueness of the infinite cluster. 

\begin{remark}\label{r:nonunique}\rm Assume that $\BM$ is discrete and that $\BQ\{p\}>0$ for each $p\in\BM$.  Given $p,q\in\BM$ we write
$p\simeq q$ if either $p=q$ or $\sup_{n\ge 1}d_\varphi^{(n)}(p,q)>0$.
It follows from \eqref{e:Dconv} that $\simeq$ is an equivalence relation.
Let $[p]:=\{q\in\BM : p\simeq q\}$ be the equivalence class of $p\in\BM$. Then $\eta_{[p]}:=\{(x,q)\in\eta: q\in [p]\}$ 
are for 
different equivalence classes independent Poisson processes with intensity measures $\lambda_d\otimes\BQ([p]\cap \cdot)$. 
Assume now that there exist some marks $p,q\in\BM$ 
such that $[p]\cap [q]=\emptyset$. We assert that $\xi[\eta_{[p]}]$ and $\xi[\eta_{[q]}]$ are vertex disjoint,
that is,  there is no edge in $\xi$ between $\eta_{[p]}$ and $\eta_{[q]}$.
To see this, we take a bounded Borel set $B\subset\R^d$ and let  $A$ denote the event that there exist $x\in B$ and $y\in\R^d$
such that $(x,p),(y,q)\in \eta$ and $(x,p)\leftrightarrow(y,q)$ in $\xi$.
Similarly, as in previous calculations, we obtain
\begin{align*}
\BP(A)\le \sum^\infty_{n=1}\lambda_d(B)\BQ\{p\}\BQ\{q\}d^{(n)}_\varphi(p,q)
\end{align*}
which comes to zero. If $\xi[\eta_{[p]}]$ and $\xi[\eta_{[q]}]$ both percolate, then $\xi$ has at least two infinite clusters.
Without the irreducibility condition, the number of infinite clusters might be any natural number or even infinity.
\end{remark}

In the following we will discuss some consequences of Theorem \ref{l:77}.
We start with the case, where  $\BQ$ has an atom.
This covers discrete (that is finite or countably infinite) mark spaces and 
generalizes Proposition \ref{p:irredunmarked}.

\begin{corollary}\label{t:irredatom}  Let $\xi$ be a stationary marked RCM and assume that 
there exists $p_0\in\BM$ with $\BQ\{p_0\}>0$.
Then $\xi$ is irreducible if and only if 
\begin{align}\label{min3}
\sup_{n\ge 1}d_\varphi^{(n)}(p_0,q)>0,\quad \BQ\text{-a.e.\ $q\in\BM$}. 
\end{align}
\end{corollary}
\begin{proof} Take $p,q\in\BM$ and $m,n\in\N$. By \eqref{e:Dconv},
\begin{align*}
d_\varphi^{(m+n)}(p,q)\ge \BQ\{p_0\}d_\varphi^{(m)}(p,p_0)d_\varphi^{(n)}(p_0,q)
=\BQ\{p_0\}d_\varphi^{(m)}(p_0,p)d_\varphi^{(n)}(p_0,q).
\end{align*}
Hence, if  \eqref{min3} holds, then we obtain that $\sup_{k\ge 2}d_\varphi^{(k)}(p,q)>0$
for $\BQ^2$-a.e.\ $(p,q)$, so that irreducibility follows from Theorem \ref{l:77}.
The converse is obvious. 
\end{proof}

Under a suitable minorization assumption on the connection function we have the following version of Corollary \ref{t:irredatom} which does not assume $p_0$ to be an atom of $\BQ$.

\begin{theorem}\label{t:irredgeneral} Assume that  
there exist a measurable set $A\subset \BM$ with $\BQ(A)>0$ and some $p_0\in \BM$
satisfying
\begin{align}\label{min1}
d_\varphi(p,q)\ge d_\varphi(p,p_0),\quad \BQ^2\text{-a.e. $(p,q)\in\BM\times A$}.
\end{align}
Assume also that \eqref{min3} holds. Then the RCM $\xi$ is irreducible.
\end{theorem}
\begin{proof} Taking $n=1$ in \eqref{e:Dconv} it follows by induction that
\begin{align*}
d^{(m)}_\varphi(p,q)\ge d^{(m)}_\varphi(p,p_0),\quad \BQ^2\text{-a.e. $(p,q)\in\BM\times A$}.
\end{align*}
Just as in the proof of Corollary \ref{t:irredatom} we hence obtain for all $m,n\in\N$
and $\BQ^2$-a.e.\ $(p,q)$ that
\begin{align*}
d_\varphi^{(m+n)}(p,q)\ge \BQ(A)d_\varphi^{(m)}(p,p_0)d_\varphi^{(n)}(p_0,q).
\end{align*}
Therefore  we obtain that $\sup_{k\ge 2}d_\varphi^{(k)}(p,q)>0$
for $\BQ^2$-a.e.\ $(p,q)$, so that irreducibility follows from Theorem \ref{l:77}.
\end{proof}

A minimal assumption for irreducibility could be 
\begin{align}\label{e:minred}
\int d_\varphi(p,q) \,\BQ(dq)>0,\quad \BQ\text{-a.e.\ $p\in\BM$}.
\end{align}
Under suitable assumptions on $\BQ$ and $\varphi$ we shall show
with Theorem \ref{t:irredPOP}
that \eqref{e:minred} is already sufficient for irreducibility.

In Theorem \ref{t:irredPOP} we will consider a partial ordering $\preceq$ on $\BM$ which is measurable, that is
$\{(p,q): p \preceq q\}$ is a measurable
subset of $\BM^2$. Slightly generalizing  \cite{Lindqvist88} we say
that $\BM$ is {\em partially ordered probability (POP) space}.
A real-valued function $f$ on $\BM$ is said to be
non-decreasing if $x\preceq y$ implies $f(x)\le f(y)$.
The probability measure $\BQ$  is called (positively) {\em associated}
if 
\begin{align}\label{e:ass}
\int fg\,d\BQ\ge \int f\,d\BQ \int g\,d\BQ 
\end{align}
for all non-decreasing measurable $f,g\colon\BM\to \R$ for which the integrals make sense.
Our next result provides assumptions on $\varphi$ and $\BQ$, under which
the minimal assumption \eqref{e:minred} implies irreducibility.
Corollary \ref{c:irredreal} and Example \ref{ex:NN} will demonstrate the usefulness of this result.

\begin{theorem}\label{t:irredPOP} Assume that $\BM$ is a POP space
and that $\BQ$ is associated. Assume also that
$d_\varphi(p,\cdot)$ is monotone (non-decreasing or non-increasing) for all $p\in\BM$.
Then the RCM $\xi$ is irreducible if and only if \eqref{e:minred} holds. 
\end{theorem}
\begin{proof} Assume that \eqref{e:lirr} holds but \eqref{e:minred} fails.
Then there exists  a measurable set $B\subset\BM$ with $\BQ(B)>0$ and
$d_\varphi(p,q)=0$ for $\BQ^2$-a.e.\ $(p,q)\in B\times \BM$. 
This implies for all $n\in\N$, that $d^{(n)}_\varphi(p,q)=0$ for $\BQ^2$-a.e.\ $(p,q)\in B\times \BM$.
The resulting contradiction shows that \eqref{e:lirr} implies \eqref{e:minred}.
Let us assume the latter holds. 
Since $\BQ$ is associated we obtain for all $p,q\in\BM$
\begin{align*}
d_\varphi^{(2)}(p,q)=\int d_\varphi(p,r)d_\varphi(r,q)\,\BQ(dr)
\ge \int d_\varphi(p,r)\,\BQ(dr) \int d_\varphi(r,q)\,\BQ(dr). 
\end{align*}
This implies \eqref{e:lirr}  and hence the result.
\end{proof}

\begin{corollary}\label{c:irredreal} Assume that $\BM\subset \R$  is an interval 
and that $d_\varphi(p,\cdot)$ is monotone  for all $p\in\BM$.
Then the RCM $\xi$ is irreducible if and only if \eqref{e:minred} holds. 
\end{corollary}
\begin{proof} Since any probability measure on $\BM$ is associated
(see e.g.\ \cite{Lindqvist88}), the result follows from Theorem \ref{t:irredPOP}.
\end{proof}

\begin{remark}\rm Let the assumptions of Theorem \ref{t:irredPOP} be satisfied
and assume moreover that $\iint\I\{p \preceq q\}\BQ(dp)\,\BQ(dq)>0$.
Let $B$ be the increasing Borel set of all $p\in\BM$ such that
$\sup_{n\ge 1}d^{(n)}(p,q)>0$  for $\BQ$-a.e.\ $q\in\BM$. 
Since $\xi$ is irreducible (by Theorem \ref{t:irredPOP}) we obtain from Theorem \ref{l:77}
that $\BQ(B)=1$.
We assert that there is some $p_0\in B$ such that $\BQ(C_{p_0})>0$, where
$C_{p_0}:=\{p\in \BX:p_0 \preceq p\}$. Indeed, if this were not the case, then
\begin{align*}
0&=\iint\I\{p \preceq q,p\in B\}\BQ(dp)\,\BQ(dq)=\iint\I\{p \preceq q,p\in B,q\in B\}\BQ(dp)\,\BQ(dq)\\
&=\iint\I\{p \preceq q\}\BQ(dp)\,\BQ(dq),
\end{align*}
contradicting our assumption. It follows that \eqref{min3} and \eqref{min1}
both hold with $A:=C_{p_0}$. Therefore Theorem \ref{t:irredgeneral} applies.
However, given a monotone $d_\varphi$ and an associated $\BQ$, 
Theorem \ref{t:irredPOP} is much easier to apply. It only remains to check
\eqref{e:minred}, which could be assumed without too much loss of generality;
see Remark \ref{r:isolated}.
\end{remark}

\begin{remark}\label{r:isolated}\rm Assume that \eqref{e:minred} fails 
and choose a measurable set $B\subset\BM$ with $\BQ(B)>0$ and
$d_\varphi(p,q)=0$ for $\BQ^2$-a.e.\ $(p,q)\in B\times \BM$. This easily implies that
\begin{align*}
|C^{(x,p)}|=1,\quad (x,p)\in\eta_{\R^d\times B},\quad \BP\text{-a.s.}.
\end{align*}
Therefore Poisson points with a mark in $B$ are isolated in $\xi$.
In particular they do not contribute to infinite clusters.
\end{remark}

\section{Deletion stability and uniqueness}\label{sec:toleranceuniqueness}

In this section, we consider a general RCM $\xi$ based on a Poisson process $\eta$ on $\BX$ with diffuse intensity measure $\lambda$.
Given $(x,\mu)\in\BX\times\bG$ 
we let 
$N^\infty(x,\mu)$ denote  the number of infinite clusters 
in $C^x(\mu)-\delta_{x}$. 
We say that the infinite clusters in $\xi$  are {\em deletion stable} if
\begin{align}\label{DT_1}
\BP(N^\infty (x,\xi^x)\ge 2)=0,\quad \lambda\text{-a.e.\ $x\in\BX$}.
\end{align}
Using the Mecke equation it is not difficult to see that
the infinite clusters in $\xi$ are deletion stable if  $N_{ds}=0$ a.s., where
\begin{align}\label{DS_2}
N_{ds}:=\int \I\{N^\infty(x,\xi)\ge 2\}\,\eta(dx). 
\end{align}

\begin{theorem}\label{t:unique} Assume that $\xi$ is irreducible and that the infinite clusters of $\xi$ are deletion stable.
Then $\xi$ has $\BP$-almost surely at most one infinite cluster.
If, conversely, the latter holds then the infinite cluster of $\xi$ is deletion stable.
\end{theorem}

The converse implication in Theorem \ref{t:unique} will be an easy consequence of the Mecke equation.
The proof of the non-trivial implication will be based on two lemmas.
Let $Y_1,\ldots,Y_n$  be random elements
of  $\BX$, which are a.s.\ pairwise distinct.
In accordance with Section \ref{s:basics}
we define a random connection model $\xi^{Y_1,\ldots,Y_n}$ based
on the point process $\eta+\delta_{Y_1}+\cdots + \delta_{Y_n}$
as follows. We connect $Y_1$ with
the points in $\eta$ using independent connection decisions which
are independent of $\xi$.
We then proceed inductively finally connecting $Y_n$ to 
$\eta+\delta_{Y_1}+\cdots + \delta_{Y_{n-1}}$.

\begin{lemma}\label{l:remove} Suppose that
$B\in\cX$ with $\lambda(B)\in(0,\infty)$ and let $Y_1,\ldots,Y_n$
be independent random variables  with distribution $\lambda_B/\lambda(B)$,
independent of $\xi$. Assume that the infinite clusters of $\xi$ are deletion stable, then
\begin{align}
\int \BP(N^\infty(Y_n,\xi^{x_1,x_2,Y_1,\ldots,Y_n})\ge 2)\,\lambda^2(d(x_1,x_2))=0. 
\end{align}
\end{lemma}
\begin{proof}  It is useful to add a point $x\in\BX$ to a graph $\mu\in\bG$
in the following explicit way. 
There are measurable mappings $\pi_n\colon\bN\to\R^d$ such
that $\mu(\cdot\times\bN)=\sum_{n=1}^{|\mu|}\delta_{\pi_n(\mu)}$,
for each $\mu\in\bG$.
Let $(\mu,x)\in\bG\times\BX$ and $u=(u_n)_{n\ge 1}\in [0,1]^\N$.
Define $\mu^x_u\in\bG$ as the graph with vertex measure $V(\mu)+\delta_x$,
edges from $\mu$ and further edges between $\pi_n(\mu)$ and $x$ if
$\varphi(\pi_n(\mu),x)\ge u_n$. Define $h(x,\mu,u):=\I\{N^\infty(x,\mu^x_u)\ge 2\}$.
Assume that $U$ is a random element of $[0,1]^\N$ with
independent and uniformly distributed components, independent of $\xi$. Then
$\I\{N^\infty(x,\xi^x)\ge 2\}$ has the same distribution as $h(x,\xi,U)$ 
and deletion stability means that
\begin{align}\label{e:tol}
\iiint h(x,\mu,u)\,\BP(\xi\in d\mu)\,\lambda(dx)\,\BP(U\in du)=0.
\end{align}
Given $x_1,x_2\in\BX$ we also have
\begin{align*}
\I\{N^\infty(Y_n,\xi^{x_1,x_2,Y_1,\ldots,Y_n})\ge 2\}\overset{d}{=}h(Y_n,\xi^{x_1,x_2,Y_1,\ldots,Y_{n-1}},U_n),
\end{align*}
where $U_n$ is independent of the pair  $(Y_n,\xi^{x_1,x_2,Y_1,\ldots,Y_{n-1}})$ and has the same distribution
as $U$. Therefore
\begin{align*}
\iint\BP&(N^\infty(Y_n,\xi^{x_1,x_2,Y_1,\ldots,Y_n})\ge 2)\,\lambda^2(d(x_1,x_2))\\
&=(\lambda(B))^{-1}\iiint \BE h(y_n,\xi^{x_1,x_2,Y_1,\ldots,Y_{n-1}},u)\,\lambda_B(dy_n)\,\BP(U\in du)\,\lambda^2(d(x_1,x_2))\\
&=(\lambda(B))^{-n}\iiint \BE h(y_n,\xi^{x_1,x_2,y_1,\ldots,y_{n-1}},u)\,\BP(U\in du)
\lambda_B^{n}(d(y_1,\ldots,y_{n}))\,\lambda^{2}(d(x_1,x_2)),
\end{align*}
where we have used the definition of $\xi^{x_1,x_2,Y_1,\ldots,Y_{n-1}}$. 
From the Mecke equation we obtain that the above equals
\begin{align*}
(\lambda(B))^{-n}\BE \iiint  h(y_n,\xi,u)\I\{y_1,\ldots,y_{n-1} \in B\}\,\lambda_B(dy_n)\,\BP&(U\in du)\\
&\eta^{(n+1)}(d(x_1,x_2,y_1,\ldots,y_{n-1})).
\end{align*}
By \eqref{e:tol}, the integral $\iint  h(y,\xi,u)\,\lambda_B(dy)\,\BP(U\in du)$ does almost
surely vanish. This concludes the proof.
\end{proof}

For given $x_1,x_2\in\BX$ let $A(x_1,x_2)$ be the event that the clusters $C^{x_1}(\xi^{x_1,x_2})$ and
$C^{x_2}(\xi^{x_1,x_2})$ are infinite and not connected.
Further, for $n\in\N_0$ let $B_n(x_1,x_2)$ be the event that $x_1$ and $x_2$ are connected 
in $\xi^{x_1,x_2,Y_1,\ldots,Y_n}$, where $Y_1,\dots,Y_n$ are defined in Lemma \ref{l:remove}.

\begin{lemma}\label{l:connected} Let the assumptions of Lemma \ref{l:remove}
be in force. Then for a given $n\in\N_0$
\begin{align}
\int \BP (A(x_1,x_2)\cap B_n(x_1,x_2))\,\lambda^2(d(x_1,x_2))=0.
\end{align}
\end{lemma}
\begin{proof} We can remove the points $Y_n,\ldots,Y_1$ from
$\xi^{x_1,x_2,Y_1,\ldots,Y_n}$ one by one. Each time we
can apply Lemma \ref{l:remove}. Hence removing $Y_i$ (for $i\le n$) cannot
split the cluster of $Y_i$ in $\xi^{x_1,x_2,Y_1,\ldots,Y_i}$
into more than one infinite cluster.
Take $x_1,x_2\in\BX$ and $n\in\N_0$ such that $B_n(x_1,x_2)$ holds and assume
for the sake of contradiction that $A(x_1,x_2)$ holds. In particular
$C^{x_1}(\xi^{x_1,x_2})$ and  $C^{x_2}(\xi^{x_1,x_2})$ are vertex disjoint, so that
there must be an $i\in\{1,\ldots,n\}$ such that $x_1,x_2$ are connected 
in $\xi^{x_1,x_2,Y_1,\ldots,Y_i}$ but not in $\xi^{x_1,x_2,Y_1,\ldots,Y_{i-1}}$.
Hence, the removal of $Y_i$ would split the cluster of $Y_i$ in $\xi^{x_1,x_2,Y_1,\ldots,Y_i}$
into two infinite clusters. This is a contradiction, showing that
almost surely $B_n(x_1,x_2)\subset A^c(x_1,x_2)$ for $\lambda^2$-a.e.\ $(x_1,x_2)\in\BX^2$.
\end{proof}

\begin{proof}[Proof of Theorem \ref{t:unique}] Let us first assume that $\BP(\xi\in A_\infty)=0$,
where $A_\infty$ is the set of all $\mu\in\bG$ such that  $\mu$ has at least 
two infinite clusters. If $x\in\eta$ satisfies $N^\infty (x,\xi)\ge 2$ then $\xi-\delta_x\in A_\infty$.
Therefore we obtain from the Mecke equation \eqref{e:Mecke}
\begin{align*}
\BE \int \I\{N^\infty (x,\xi)\ge 2\}\,\eta(dx)\le \BE \int \I\{\xi-\delta_x\in A_\infty\}\,\eta(dx)
=\BE \int \I\{\xi\in A_\infty\}\,\lambda(dx)
\end{align*}
which comes to zero.

Let us now assume that $\xi$ is irreducible and that the infinite clusters are
deletion stable.
We need to show that almost surely two 
points of $\eta$ cannot belong to two different infinite clusters. By the Mecke equation 
\eqref{e:Mecke} for $n=2$ the latter is equivalent to
\begin{align}
\int \BP (A(x_1,x_2))\,\lambda^2(d(x_1,x_2))=0.
\end{align}
The following arguments apply to $\lambda^2$-a.e.\ $(x_1, x_2)\in\BX^2$. By Proposition \ref{p:irred} (vi)
there exist a set $B\in\cX$ with $0<\lambda(B)<\infty$, an $n\in\N$ and random variables $Y_1,\ldots,Y_n$ 
with distribution $\lambda_B/\lambda(B)$ such that $\BP(B'_n(x_1, x_2))>0$, where
\begin{align*}
    B'_n(x_1, x_2):=\{x_1\leftrightarrow x_2 \text{ in } \Xi'[x_1,x_2,Y_1,\ldots,Y_n]\}.
\end{align*}
We can couple the random graphs $\xi^{x_1,x_2,Y_1,\ldots,Y_{n}}$ and
$\Xi'[x_1,x_2,Y_1,\ldots,Y_n]$ in such a way that $\xi^{x_1,x_2}$ and $\Xi'[x_1,x_2,Y_1,\ldots,Y_n]$
are independent and every edge in the latter graph
is also present in the former. Then $B'_n(x_1,x_2)$ implies $B_n(x_1,x_2)$ and we obtain 
from Lemma \ref{l:connected}  that 
\begin{align*}
    \BP(A(x_1,x_2)\cap B'_n(x_1,x_2))=0.
\end{align*}
By the above coupling the events $A(x_1,x_2)$  and $B'_n(x_1,x_2)$ are independent.
Hence $\BP(A(x_1,x_2))=0$, as required.
\end{proof}

Motivated by \cite{NashWilliams} we might call an infinite graph $\mu\in\bG$ {\em $2$-indivisible} if
the removal of a finite number of vertices results in at most 
one infinite connected component. The 
following corollary of Proposition \ref{p:deltolerant} shows that
$\xi$ is almost surely $2$-indivisible.

\begin{corollary}\label{c:necess}\rm
Assume that $\xi$  has almost surely at most one infinite cluster. Assume further
that $\eta_0$ is a point process such that $\BP(\eta_0(\BX)<\infty)=1$ and $\BP(\eta_0\le\eta)=1$. Then
$\xi[\eta-\eta_0]$ has a.s.\ at most one infinite cluster.
\end{corollary}

\begin{remark}\label{r:nondiffuseunique}\rm Consider the setting of Remark \ref{r:non-diffuse}.
Assume that the infinite clusters of $\xi$ are deletion stable.
This means that reducing the size
of an atom $x$ of $\eta$ by one (and removing the corresponding edges), cannot split the associate cluster into more
than one infinite clusters. 
This is not an intrinsic property of the random graph $\xi^*$, as defined in Remark \ref{r:non-diffuse}.
Assume in addition that $\hat\xi$ is irreducible (characterized in Remark \ref{r:non-diffuseirred}), so 
that Theorem  \ref{t:unique} applies. By Corollary \ref{c:necess} we can then remove any finite number
of points from  $\hat\eta$ without splitting the infinite cluster (if existent) in two or more infinite clusters.
In particular we can remove a finite number of points from the support of $\eta$,
without splitting the infinite cluster of $\xi^*$ in two or more infinite clusters.
\end{remark}

\begin{remark}\label{r:redpoints}\rm In accordance with the physics literature
(see e.g.\ \cite{BundeHavlin12}) we might call a point $x\in\eta$ {\em red}, if any doubly infinite path
in $\xi$ has to use $x$. If $\xi$ has a unique infinite cluster, then Corollary \ref{c:necess} says
in particular that $\xi$ cannot have red points. 
More generally, we may call
a subset of $\eta$ {\em red}, if any doubly infinite path in $\xi$ contains at least one point from this set. 
Corollary \ref{c:necess} says that $\xi$ cannot have a finite red set.
\end{remark}

\begin{remark}\label{r:BoJaRi}\rm The authors of \cite{BoJaRi07}
studied random connection models on finite point processes in an asymptotic setting.
Under a natural irreducibility assumption (similar to Proposition \ref{p:irred} (iv))
they  proved uniqueness of the giant component;
see Theorem 3.6 and Example 4.9 in \cite{BoJaRi07}.
\end{remark}

\begin{remark}\label{r:hyperbolic}\rm Consider Example \ref{e:generalBoolean}
in the special case where $\BY$ is the hyperbolic plane and $\nu=t\mu_2$, where $\mu_2$ is the 
invariant measure on $\BY$ and $t>0$ is an intensity parameter. Let $\BQ$ be concentrated
on a single positive radius and let $\xi$ be RCM with connection function
$\varphi(K,L)=\I\{K\cap L\ne\emptyset\}$. Then $\xi$ describes a hyperbolic Boolean model of balls.
Inspired by the seminal paper \cite{BenjSchramm01},
it was shown in \cite{Tykesson07} that there are numbers $0<t_c<t_u<\infty$ such
that there is no percolation for $t\in(0,t_c]$, infinitely many infinite clusters for 
$t\in(t_c,t_u)$ and a unique infinite cluster for $t\ge t_u$. Corollary \ref{c:necess}
shows that the unique infinite cluster cannot be destroyed by the removal of  a finite number of
Poisson points. A referee of this paper has asked whether it is possible to check (resp.\ to reject) 
deletion stability for certain values of $t$ which are then (by Theorem \ref{t:unique}) upper 
(resp.\ lower) bounds for $t_u$. We do not know whether this is possible; see also 
Remark \ref{r:amenability}.
\end{remark}

\section{A spatial Markov property}\label{secMarkov}

We again consider a general RCM $\xi$ based on a Poisson process $\eta$ on $\BX$ with diffuse intensity measure $\lambda$.
Let $v\in\BX$. In the next section  
we shall establish and exploit a useful explicit change of measure for the distribution of $C^v=C^v(\xi^v)$.
This is possible since  for $n\in\N_0$ the conditional distribution of $C^v_{n+1}$ given $C^v_{\le n}$
can be described in terms of a RCM driven by Poisson process with a thinned intensity measure.
In this section we derive a general version of this {\em spatial Markov property}.

Let $\nu$ be a locally finite and diffuse
measure on $\BX$. Then we denote by $\Pi_\nu$ the distribution
of a Poisson process with this intensity measure.
We define two kernels from $\bN$ to $\BX$ and from $\bN\times\bN$ to $\BX$ (using the same notation $K_\nu$ for simplicity), by 
\begin{align}
K_\nu(\mu,dx):=\bar\varphi(\mu,x)\nu(dx), \quad K_\nu(\mu,\mu',dx):=\bar\varphi(\mu,x)\varphi(\mu',x)\nu(dx),
\end{align} 
where we recall the definitions \eqref{e:barvarphi}. 
Proposition \ref{p2.1} will provide an interpretation of this kernel. 
Denoting by $0$ the zero measure, we note that
\begin{align}
K_\nu(0,dx)=\nu(dx), \quad K_\nu(0,\mu',dx)=\varphi(\mu',x)\nu(dx), \quad K_\nu(\mu,0,dx)=0.
\end{align}
We write $K_\nu(\mu):=K_\nu(\mu,\cdot)$ and $K_\nu(\mu,\mu'):=K_\nu(\mu,\mu',\cdot)$.
Note that $K_\lambda(0,\mu,\BX)=\varphi_\lambda(\mu)$; see \eqref{e:barvarphi}.

For  $n\in\N_0$, $\mu\in\bG$ and $v\in\BX$ let $\Gamma^v_n(\lambda,\mu,\cdot)$ denote the
distribution of a random graph $\xi_n$ defined as follows.
Let $\xi_n'$ be a RCM based on $\eta_n+C^v_{n}(\mu)$,
where $\eta_n$ is a Poisson process with intensity measure $K_\lambda(C^v_{\le{n-1}}(\mu))$, 
and where we recall that $C^v_{\le{-1}}:=0$.
Remove in $\xi_n'$ all edges between
vertices from $C^v_{n}(\mu)$ to obtain a random graph $\xi_n''$. Finally
set $\xi_n:=C^v_{\le n}(\mu)\oplus \xi_n''$, with an obvious definition of the operation $\oplus$.
We set $C^v_{\le 0}(\mu):=\delta_v$, which is the graph with vertex set $\{v\}$ and
no edges.

\begin{theorem}\label{spatMarkov} Let $v\in\BX$ and $n\in\N_0$. Then,
\begin{align}
\BP(\xi^v\in\cdot\mid C^v_{\le n})=\Gamma^v_n(\lambda,C^v_{\le n},\cdot),\quad \BP\text{-a.s.}
\end{align}
\end{theorem}
\begin{proof} This follows from the proof of \cite[Lemma 3.3]{HHLM23};
see also Proposition 2 in \cite{MeePenSar97}.
Essentially
the assertion is equivalent to equation (3.6) in this proof. The arguments given there apply to a RCM on a general state space $\BX$ and not only to $\R^d$.
\end{proof}

A quick consequence of Theorem \ref{spatMarkov} is that
$\{(V(C^v_{\le{n-1}}),V(C^v_n))\}_{n\in\N_0}$ is a Markov process.

\begin{proposition}\label{p2.1} The sequence $\{(V(C^v_{\le{n-1}}),V(C^v_n))\}_{n\in\N_0}$ is a Markov
process with transition kernel
\begin{align*}
(\mu,\mu')\mapsto \int\I\{(\mu+\mu',\psi)\in\cdot\}\,\Pi_{K_\lambda(\mu, \mu')}(d\psi).
\end{align*}
\end{proposition}

We also note that
\begin{align*}
K_{\lambda}(\mu,\mu',\BX)\le \int \varphi(\mu',x)\,\lambda(dx)
\le \iint \varphi(y,x)\,\mu'(dy)\,\lambda(dx),
\end{align*}
where we have used the Bernoulli inequality. Hence
\begin{align}\label{e3.7a} 
K_{\lambda}(\mu,\mu',\BX)\le\int D_\varphi(y)\,\mu'(dy).
\end{align}

\begin{corollary}\label{c3.23} Let $n\in\N_0$. Then we have  for $\lambda$-a.e.\ $v\in\BX$ that
$\BP(|C^v_n|<\infty)=1$.
\end{corollary}
\begin{proof} We can proceed by induction. 
For $n=0$ the assertion is trivial.
Assume that  $\BP(|C^v_n|<\infty)=1$ for some $n\in\N_0$.
From Proposition \ref{p2.1} we know that
the conditional distribution of $V(C^v_{n+1})$ given 
$(V(C^v_{\le{n-1}}),V(C_n^v))$ is that of a Poisson process
with intensity measure $K_\lambda(V(C^v_{\le{n-1}}),V(C_n^v))$.
By \eqref{e3.7a} we obtain that 
\begin{align*}
\BE[|C^v_{n+1}|\mid (V(C^v_{\le{n-1}}),V(C_n^v))]\le \int D_\varphi(y)\,C_n^v(dy) 
\end{align*}
which is for $\lambda$-a.e.\ $v\in\BX$ a.s.\ finite by our general assumption \eqref{evarphiintlocal}
and induction hypothesis.
\end{proof}

The following useful property of the kernel $K_\lambda$ can easily be proved by induction.

\begin{lemma}\label{l3.1} Let  $n\in\N$ and $\mu_0,\ldots,\mu_n\in\bN$.
Then
\begin{align*}
K_\lambda(0,\mu_0)+K_\lambda(\mu_0,\mu_1)+\cdots+K_\lambda(\mu_0+\cdots+\mu_{n-1},\mu_{n})
=K_\lambda(0,\mu_0+\cdots+\mu_{n}).
\end{align*}
\end{lemma}

\section{Perturbation formulas}\label{secperturb}

In the next sections we vary the intensity measure $\lambda$ and
consider $t\lambda$ for $t\in\R_+$. We fix $v\in\BX$ 
and let $\BP_t$ be a probability
measure governing a RCM $\xi$ based on $\eta$, where $\eta$
is a Poisson process with intensity measure $t \lambda$. 
The associated expectation is denoted by $\BE_t$.
Recall the definition \eqref{e:barvarphi}.

\begin{lemma}\label{l:abscont}
Let $\tilde\xi$ be a RCM based on a Poisson process $\tilde\eta$ with finite
intensity measure $\nu$. Let $f\colon\bG\to[0,\infty)$. Then
\begin{align*}
\BE_tf(\tilde\xi)=\BE_1f(\tilde\xi)t^{|\tilde\eta|}e^{(1-t)\nu(\BX)}
\end{align*}
\end{lemma}
\begin{proof} It is well-known that
\begin{align}\label{eabsc}
\Pi_{t\nu}=\int \I\{\mu\in\cdot\}t^{|\mu|} e^{(1-t)\nu(\BX)}\,\Pi_{\nu}(d\mu),\quad t\ge 0.
\end{align}
This follows, for instance from \cite[Exercise 3.7]{LastPenrose18}
and an easy calculation. The assertion then follows by conditioning, using
the kernel $\Gamma$ in \eqref{e:kernel}.
\end{proof}

\begin{proposition}\label{propert0} Let $v\in\BX$, $t\in\R_+$, $n\in\N$ and $t_0>0$. Then 
\begin{align*}
\BP_t(C^v_{\le n}\in\cdot)
=\BE_{t_0}\I\{C^v_{\le n}\in\cdot\}(t/t_0)^{\big|C^v_{\le n}\big|-1}e^{(t_0-t)\varphi_{\lambda}\big(C^v_{\le{n-1}}\big)}.
\end{align*}  
\end{proposition}
\begin{proof} 
It is sufficient to consider the special case $t_0=1$.
The general case can be proved similarly or can be derived from the special case.
We omit the dependence on $v$ in our notation by writing $C_n:=C^v_n$,
and $C_{\le n}:=C^v_{\le n}$. 
Given $\mu\in\bG$ we let $C_n^+(\mu)$ denote the graph $\mu[V(C_{n-1}(\mu))+V(C_n(\mu))]$
with the edges between vertices of $C_{n-1}(\mu)$ removed.

Let $f\colon\bG\to[0,\infty)$ be measurable.
By Theorem \ref{spatMarkov},
\begin{align*}
\BE_t f(C_{\le n})=\BE_t\int f(C_{\le{n-1}}\oplus C_n^+(\mu))\,\Gamma_{n-1}(t\lambda,C_{\le{n-1}},d\mu).
\end{align*}
By \eqref{e3.7a}  we have
\begin{align*}
K_{t\lambda}(C_{\le{n-2}},C_{n-1},\BX)\le  t \int D_\varphi(y)\, C_{n-1}(dy)
\end{align*}
which is almost surely finite by Corollary \ref{c3.23}.
By definition of $\Gamma_n$ and the thinning properties of a Poisson
process, the distribution of $C^+_n(\cdot)$ under
$\Gamma_{n-1}(t\lambda,C_{\le{n-1}},\cdot)$ is that of a RCM 
driven by a Poisson process with intensity measure
$K_{t\lambda}(C_{\le{n-2}},C_{n-1})$
with additional independent connections to $V(C_{n-1})$; see also Proposition \ref{p2.1}.
Therefore we obtain from Lemma \ref{l:abscont} that
\begin{align*}
\BE_t f(C_{\le n})=\BE_t\int f(C_{\le{n-1}}\oplus C_n^+(\mu))
e^{(1-t)K_{\lambda}(C_{\le{n-2}},C_{n-1},\BX)}t^{|\mu|}
\,\Gamma_{n-1}(\lambda,C_{\le{n-1}},d\mu).
\end{align*}
Iterating this identity yields that the above equals
\begin{align*}
\idotsint f(C_1^+(\mu_1)&\oplus \cdots\oplus C_n^+(\mu_n))
e^{(1-t)K_{\lambda}(\delta_v+\mu_1+\cdots+\mu_{n-2},\mu_{n-1},\BX)}\cdots
e^{(1-t)K_{\lambda}(\delta_v,\mu_1,\BX)}\\
&t^{|\mu_n|}\cdots t^{|\mu_1|}
\Gamma_{n-1}(\lambda,C_1^+(\mu_1)\oplus \cdots\oplus C_{n-1}^+(\mu_{n-1}),d\mu_n)
\cdots \Gamma_0(\lambda,\delta_v,d\mu_1).
\end{align*}
By Lemma \ref{l3.1} this equals
\begin{align*} 
\idotsint &f(C_1^+(\mu_1)\oplus \cdots\oplus C_n^+(\mu_n))
e^{(1-t)K_{\lambda}(0,\delta_v+\mu_1+\cdots+\mu_{n-1},\BX)}
t^{|\mu_1|+\cdots+|\mu_n|}\\
&\Gamma_{n-1}(\lambda,C_1^+(\mu_1)\oplus \cdots\oplus C_{n-1}^+(\mu_{n-1}),d\mu_n)
\cdots \Gamma_0(\lambda,[\delta_v],d\mu_1).
\end{align*}
By Theorem \ref{spatMarkov} 
we obtain
\begin{align*}
\BE_t f(C_{\le n})=\BE_1 f(C_{\le n})t^{|C_{\le n}|-1}e^{(1-t)\varphi_{\lambda}(C_{\le{n-1}})}
\end{align*}
and hence the assertion.
\end{proof}

\begin{theorem}\label{tabsc} Let $v\in\BX$, $t\in\R_+$ and $t_0>0$. Then
\begin{align}
\BP_t(C^v\in\cdot,|C^v|<\infty)=\BE_{t_0} \I\{C^v\in\cdot,|C^v|<\infty\} (t/t_0)^{|C^v|-1} e^{(t_0-t)\varphi_{\lambda}(C^v)}.
\end{align}
\end{theorem}
\begin{proof} Again it is sufficient to consider the special case $t_0=1$.
By Proposition \ref{propert0} the distribution 
$\BP_t(C^v_n\in\cdot)$ is absolutely continuous w.r.t.\ 
$\BP_1(C^v_n\in\cdot)$ with Radon--Nikodym derivative
$M^v_n:=t^{|C^v_{\le n}|-1}e^{(1-t)\varphi_{\lambda}(C^v_{\le{n-1}})}$.
In particular $\{M^v_n\}_{n\in\N_0}$ is a (non-negative) martingale with respect to 
$\{\sigma(C^v_{\le n})\}_{n\in\N_0}$
and therefore converges a.s.\ towards  $M^v_\infty:=\limsup_{n\to\infty}M^v_n$.
By \cite[Theorem VII.6.1]{Shir} we have
\begin{align*}
\BP_t(C^v\in\cdot)
=\BE_1 \I\{C^v\in\cdot\} M^v_\infty+\BE_t\I\{C^v\in\cdot,M^v_\infty=\infty\}.
\end{align*}
On the event $\{|C^v|<\infty\}$ we clearly have
\begin{align*}
M^v_\infty= t^{|C^v|-1}e^{(1-t)\varphi_{\lambda}(C^v)}
\end{align*}
which is finite. This concludes the assertion.
\end{proof}

Let $f\colon\bG\to\R$ be a measurable mapping.
Define for $v\in\BX$, $n\in\N$, and $t\in\R_+$
\begin{align}\label{f_F_mom}
   F_n(\xi^v):=f(C^v)\I\{|C^v|=n\}, \ & \  f^{v}_n(t):=\BE_t F_n(\xi^v), \\
    F_{\le n}(\xi^v):=  f(C^v)\I\{|C^v|\le n\}, \ & \ f^{v}_{\le n}(t):= \BE_t  F_{\le n}(\xi^v),\\
    F(\xi^v):=  f(C^v)\I\{|C^v|<\infty\}, \ & \ f^{v}(t):= \BE_t  F(\xi^v).
\end{align}
We also write
$|f|^v(t):=\BE_t |F(\xi^v)|$ and define $|f|^v_n(t)$ and
$|f|^v_{\le n}(t)$ similarly. We are interested in the analytic properties of
the function $f^v(t)$ under the assumption $|f|^v(t)<\infty$. A key
example is the {\em position dependent cluster density}
\begin{align}\label{FE}
\kappa^v(t):=\BE_t |C^v|^{-1},\quad t\in\R_+.
\end{align}
Our terminology is motivated by the stationary marked case
(see Lemma \ref{l:kappa}) and also supported by the Mecke equation,
implying 
\begin{align*}
\int t \kappa^v(t)\,\lambda_B(dv)= \BE_t \int|C^v(\xi)|^{-1}\,\eta_B(dv),\quad B\in\cX.
\end{align*}

Suppose that $|f|^v_n(t_0)<\infty$ for some $t_0>0$ and $n\in\N$, then by Theorem \ref{tabsc}
\begin{align}\label{f_n(t)}
f^{v}_{n}(t)=\left(\frac{t}{t_0}\right)^{n-1}\int_0^\infty e^{-tu}\, \nu_{f,n,t_0}(du),
\end{align}
where the signed measure $\nu_{f,n,t_0}$ is defined by
\begin{align}\label{nun_f}
\nu_{f,n,t_0}(\cdot):=\BE_{t_0} \I\{\varphi_\lambda(C^v)\in \cdot\} e^{t_0 \varphi_\lambda(C^v)} F_n(\xi^v).
\end{align}
By Corollary \ref{c3.23} this is a locally finite signed measure on
$\R_+$.  
It follows from Theorem \ref{tabsc} that the function
$|f|^{v}_n(t)/t^{n-1}$ is monotone decreasing on $(0,\infty)$, so that
\begin{align*}
    |f|^{v}_n(t)\le \left(\frac{t}{t_0}\right)^{n-1} |f|^{v}_n(t_0),\quad t\ge t_0.
\end{align*}

\begin{lemma}\label{l:f_n_1} Let $v\in\BX$, $n\in\N$ and $t_0>0$. If $|f|^{v}_n(t_0)<\infty$, then  for $t\ge t_0$
\begin{align*}
f_n^{v}(t)=\frac{t^n}{t_0^{n-1}}\int_0^\infty \nu_{f,n,t_0}[0,u]e^{-t u}\, du.
\end{align*}
\end{lemma}
\begin{proof}  We obtain from \eqref{f_n(t)}
that
\begin{align}
f^{v}_n(t)=\frac{t^n}{t_0^{n-1}} \iint \I\{u\le s\}e^{-ts}\,ds\,\nu_{f,n,t_0}(du).
\end{align}
Since $\nu_{f,n,t_0}$ is locally finite, we can apply  Fubini's theorem to obtain the assertion.
\end{proof}

\begin{lemma}\label{l:f_n_2}
Let $v\in\BX$, $n\in\N$ and $t_0>0$. If $|f|^{v}_n(t_0)<\infty$, then the function $f^{v}_n$ 
is analytic on $(t_0,\infty)$ and  for $t>t_0$
\begin{align}\label{fndiff}
\frac{d}{dt}f_n^{v}(t)=\frac{nt^{n-1}}{t_0^{n-1}}\int_0^\infty \nu_{f,n,t_0}[0,u]e^{-t u}\, du
-\frac{t^n}{t_0^{n-1}}\int_0^\infty u \nu_{f,n,t_0}[0,u]e^{-t u}\, du.
\end{align}
\end{lemma}
\begin{proof} Let
$\Omega_{t_0}:=\{z \in \mathbb{C}: \Re(z)>t_0\}$, 
and extend $f^{v}_n$ to $\Omega_{t_0}$ by setting 
\begin{align*}
    f^{v}_{n}(z):=\frac{z^n}{t_0^{n-1}} \int_0^\infty \nu_{f,n,t_0}[0,u] e^{-z u}\, d u,\quad z\in \Omega_{t_0}. 
\end{align*}
By \eqref{nun_f} we have
\begin{align*}
   |\nu_{f,n,t_0}[0,u]|\le e^{t_0 u} |f|^{v}_{n}(t_0).
\end{align*}
Since $|f|^{v}_{n}(t_0)<\infty$
this implies that $f^{v}_{n}$ is a complex analytic function on $\Omega_{t_0}$. Since $(t_0,\infty) \subset$ $\Omega_{t_0} \cap \R$, 
the restriction of this function to $(t_0,\infty)$ is real analytic. The formula \eqref{fndiff}
follows from  Lemma \ref{l:f_n_1} the product rule of calculus and the Leibniz rule for differentiating integrals.
The latter can be applied since for each $\varepsilon>0$ and all $u>0$
\begin{align*}
u   |\nu_{f,n,t_0}[0,u]|e^{-tu}\le u e^{(t_0 -t)u} |f|^{v}_{n}(t_0)\le u e^{-\varepsilon u} |f|^{v}_{n}(t_0) 
\end{align*}
uniformly for $t\ge t_0+\varepsilon$.
\end{proof}

To rewrite Lemma \ref{l:f_n_2} in a different way, we define
\begin{align}    
M^v_t:=|C^v|-1-t \varphi_{\lambda}(C^v),\quad t\in\R_+,\,v\in\BX.
\end{align}

\begin{lemma}\label{l:f_n_3b} 
Let $v\in\BX$, $n\in\N$ and $t_0>0$. If $|f|^{v}_n(t_0)<\infty$, then
function $f^{v}_n$ is analytic on $(t_0,\infty)$ and  for $t>t_0$
    \begin{align}\label{fndiff2}
        \frac{d}{dt}f_n^{v}(t)=t^{-1} \BE_t \big[ M^v_t F_n(\xi^v) \big]. 
    \end{align}
\end{lemma}
\begin{proof} By Theorem \ref{tabsc},
\begin{align*}
f_n^{v}(t)=\left(\frac{t}{t_0}\right)^{n-1}\BE_{t_0}\big[F_n(\xi^v)e^{(t_0-t)\varphi_{\lambda}(C^v)}\big].
\end{align*}
Hence the result follows from Lemma \ref{l:f_n_2} and calculus, where the application of the Leibniz differentiation rule
can be justified as in the proof of Lemma \ref{l:f_n_2}.
\end{proof}

\begin{lemma}\label{l:f_n_3}
Let $v\in\BX$, $n\in\N$ and $t_0>0$. If $|f|^{v}_{\le n}(t_0)<\infty$, then 
function $f^{v}_{\le n}$ 
is analytic on $(t_0,\infty)$ and for $t>t_0$
\begin{align}\label{Fndiff}
\frac{d}{dt}f_{\le n}^{v}(t)
=t^{-1} \BE_t \big[ M^v_t F_{\le n}(\xi^v)\big].
\end{align}
\end{lemma}
\begin{proof}
The result follows from the definition of $f^v_{\le n}$ and Lemma \ref{l:f_n_3b}, since $|f|^{v}_{\le n}(t_0)=\sum\nolimits_{k=1}^n |f|^v_k(t_0)$.
\end{proof}

\begin{theorem}\label{t:F_deriv} Let $0<t_0<t_1<\infty$. Assume for each $t\in[t_0,t_1]$
that $|f|^v(t)<\infty$. Assume moreover that for each $\varepsilon>0$,

\begin{align}\label{uniform_conv_1}
\lim_{n\to\infty}\sup_{t\in[t_0+\varepsilon,t_1]}\left|\sum_{k>n} \frac{d}{dt}f_{k}^v(t)
\right|=0.
\end{align}
Then $f^v$ is continuously differentiable on $(t_0,t_1]$
with derivative given by
\begin{align}\label{e:derivative}
  \frac{d}{dt} f^v(t)
=\lim_{n\to\infty}t^{-1}\frac{d}{dt} f^v_{\le n}(t)=t^{-1}\sum_{n=1}^\infty\frac{d}{dt} f^v_{n}(t).
\end{align}
\end{theorem}  
\begin{proof} Let $n\in\N$.
Since $|f|^{v}_{\le n}(t_0)\le |f|^v(t_0)<\infty$, we can apply Lemma \ref{l:f_n_3}
to obtain that the function $f^{v}_{\le n}$ is analytic on $(t_0,\infty)$, with derivative  
\begin{align*}
\frac{d}{dt}f_{\le n}^{v}(t)=t^{-1}\sum_{k=1}^n \frac{d}{dt}f_n^{v}(t)=t^{-1} \BE_t \big[M^v_t F_{\le n}(\xi^v)\big], \quad t>t_0.
\end{align*}
By dominated convergence
\begin{align*}
\lim_{n\to\infty}f^{v}_{\le n}(t)=f^{v}(t).
\end{align*}

Furthermore we have
\begin{align*}
\left|\frac{d}{dt} f^v(t)-\frac{d}{dt}f_{\le n}^v(t)\right|=t^{-1}\left|\sum_{k> n}\frac{d}{dt}f^v_k(t) \right|.
\end{align*}
By assumption \eqref{uniform_conv_1} this tends to zero uniformly
in $t\in[t_0+\varepsilon,t_1]$ for each $\varepsilon>0$.
A standard result of analysis gives us that
$f^v$ is continuously differentiable on $(t_0,t_1]$ with derivative given by the right-hand side of 
\eqref{e:derivative}. 
\end{proof}

\begin{theorem}\label{t:F_deriv2}
     Let $0<t_0<t_1<\infty$. Assume for each $t\in[t_0,t_1]$
that $|f|^v(t)<\infty$. Assume moreover that for each $\varepsilon>0$,
\begin{align}\label{uniform_conv_2}
\lim_{n\to\infty}\sup_{t\in[t_0+\varepsilon,t_1]}\sum_{k>n} \BE_t\big|M^v_t F_{k}(\xi^v)\big|=0.
\end{align}
Then $f^v$ is continuously differentiable on $(t_0,t_1]$
with derivative given by
\begin{align}\label{e:derivative_2}
\frac{d}{dt} f^v(t)=t^{-1}\BE_t \big[M^v_t F(\xi^v)\big].
\end{align}
\end{theorem}
\begin{proof} 
Let $n\in\N$.
Since $|f|^{v}_{n}(t_0)\le |f|^{v}_{\le n}(t_0)\le |f|^{v}(t_0)<\infty$, then by Lemma \ref{l:f_n_2}
the function $f^{v}_{n}$ is analytic on $(t_0,\infty)$, with derivative  
\begin{align*}
\left|\frac{d}{dt}f_{ n}^v(t)\right|=t^{-1}\left|\BE_t \big[ M^v_t F_{n}(\xi^v)\big]\right|\le t^{-1}\BE_t  \big|M^v_t F_{n}(\xi^v)\big|, \quad t>t_0.
\end{align*}
 Hence
\begin{align*}
    \left|\sum_{k>n} \frac{d}{dt}f_{k}^v(t)
\right| \le t^{-1}\sum_{k>n}\BE_t \big|M^v_t F_{n}(\xi^v)\big|.
\end{align*}
By assumption \eqref{uniform_conv_2} this tends to zero uniformly
in $t\in[t_0+\varepsilon,t_1]$ for each $\varepsilon>0$. Therefore by Theorem \ref{t:F_deriv} the function $f^v(t)$ is continuously differentiable on $(t_0,t_1]$ with derivative 
\begin{align*}
    \frac{d}{dt} f^v(t)
=t^{-1}\sum_{n=1}^\infty\frac{d}{dt} f^v_{n}(t) = t^{-1}\BE_t \big[M^v_t F(\xi^v)\big],
\end{align*}
where the last equality we get from Fubini's theorem, since by assumption \eqref{uniform_conv_2} we have that for $t\in(t_0,t_1]$
\begin{align}\tag*{\qedhere}
   \BE_t\big|M^v_t F(\xi^v)\big|=\sum_{n\ge 1}\BE_t \big|M^v_t F_{n}(\xi^v)\big|<\infty.
\end{align}

\end{proof}

The following theorem provides a large class of functions
satisfying the assumptions of Theorem \ref{t:F_deriv2}, covering the
cluster density \eqref{FE}.
We shall prove it in Section \ref{s:diffenergy}.

\begin{theorem}\label{t:F_deriv3}  
Let $f\colon\bG\to\R$ be a measurable mapping satisfying
$|f(\mu)|\le |\tilde{f}(|V(\mu)|)|$ for each $\mu\in\bG$,
where $\tilde{f}:\N\to\R$ satisfies
\begin{align}\label{prop_f_7}
    \lim_{n\to\infty}\tilde{f}(n)\sqrt{n\log n}=0.
\end{align}
Then $f^v$ is for each $v\in\BX$ continuously differentiable on $(0,\infty)$ with derivative given by \eqref{e:derivative_2}.
\end{theorem}

\section{Difference operators}\label{secdiffop}

In this section we shall rewrite Theorem \ref{t:F_deriv} and Theorem \ref{t:F_deriv2}
in the form of a Margulis--Russo formula. Recall that $\mu-\delta_x:=\mu[V(\mu)-\delta_x]$ is
the graph resulting from $\mu$ by removing the point $x$
(if $x\in V(\mu)$) along with all edges with vertex $x$  for $\mu\in\bG$ and $x\in\BX$.
Given a measurable function $f\colon\bG\to\R$ 
and $x\in\BX$ we define $\nabla_xf\colon\bG\to\R$ by
\begin{align}
\nabla_xf(\mu):=f(\mu)-f(\mu-\delta_x).
\end{align}

\begin{theorem}\label{t:MR} 
 Let the assumptions of Theorem \ref{t:F_deriv} be satisfied.
Then the function $f^v$ is continuously differentiable on $(t_0,t_1]$
with derivative given by
\begin{align}\label{e:derivative_3}
\frac{d}{dt}f^v(t)=\lim_{n\to\infty}t^{-1}\BE_t \int \nabla_x F_{\le n}(\xi^v)\,C^{v!}(dx)=\sum_{n\ge 1} t^{-1}\BE_t \int \nabla_x F_{n}(\xi^v)\,C^{v!}(dx).
\end{align}
\end{theorem}  

We start the proof with the counterpart of Lemma \ref{l:f_n_2} and Lemma \ref{l:f_n_3}. 

\begin{lemma}\label{l:f_n_MR}
Let $v\in\BX$, $n\in\N$ and $t_0>0$. If $|f|^{v}_{n}(t_0)<\infty$, then function $f^{v}_{n}$ is analytic on $(t_0,\infty)$ 
and for $t>t_0$
\begin{align*}
\frac{d}{dt}f_{n}^{v}(t)=t^{-1} \BE_t \int  \nabla_x F_{ n}(\xi^v)\,C^{v!}(dx).
\end{align*}
\end{lemma}
\begin{proof}
Let $t>t_0$. We wish to apply Lemma \ref{l:f_n_3b}.  By definition we have
\begin{align*}
\BE_t |F_n(\xi^v)||C^v|=n|f|^v_n(t)\le  n\left(\frac{t}{t_0}\right)^{n-1} |f|^v_n(t_0)
\end{align*}
which is finite by assumption. Therefore we obtain from Theorem \ref{tabsc},
\eqref{nun_f} (with $|f|$ instead of $f$) and Fubini’s theorem 
\begin{align*}
        \BE_t |F_n(\xi^v)|\varphi_\lambda(C^v)=& \left(\frac{t}{t_0}\right)^{n-1} \int_0^\infty u e^{-tu} \, \nu_{|f|,n,t_0}(du)\\
        =&\left(\frac{t}{t_0}\right)^{n-1} \int_0^\infty\int_u^\infty (ts-1) e^{-ts}\, ds\, \nu_{|f|,n,t_0}(du)\\
        =&\left(\frac{t}{t_0}\right)^{n-1} \int_0^\infty \nu_{|f|,n,t_0}[0,s] (ts-1) e^{-ts}\, ds \\
\le & \, |f|^v_n(t_0)  \frac{t^n}{t_0^{n-1}}\int_0^\infty  s e^{(t_0-t)s}\, ds<\infty, 
    \end{align*}
where we have used that $\nu_{|f|,n,t_0}[0,s]\le e^{t_0s}|f|^v_n(t_0)$.
Hence we obtain from Lemma \ref{l:312} that
\begin{align*}
t \,\BE_t F_{n}(\xi^v) \varphi_\lambda(C^v)=\BE_t\int F_{ n}(\xi^v-\delta_x)\,C^{v!}(dx).
\end{align*}
Now the assertion follows from Lemma \ref{l:f_n_3b}.
\end{proof}

\begin{lemma}\label{l:f_n_MR_2}
Let $v\in\BX$, $n\in\N$ and $t_0>0$. If $|f|^{v}_{\le n}(t_0)<\infty$, then function $f^{v}_{\le n}$ is analytic on $(t_0,\infty)$ and for $t>t_0$
\begin{align*}
\frac{d}{dt}f_{\le n}^{v}(t)=t^{-1} \BE_t \int  \nabla_x F_{\le n}(\xi^v)\,C^{v!}(dx).
\end{align*}
\end{lemma}
\begin{proof}
The result follows from the definition of $f^{v}_{\le n}$, Lemma \ref{l:f_n_3} and Lemma \ref{l:f_n_MR}.
\end{proof}

\begin{proof}[Proof of Theorem \ref{t:MR}] 
By Theorem \ref{t:F_deriv} we have that $f^v$ is continuously differentiable on $(t_0,t_1]$ with derivative given by 
\eqref{e:derivative}. Hence we can apply Lemma \ref{l:f_n_MR} and Lemma \ref{l:f_n_MR_2} to obtain the assertion. 
\end{proof}

\begin{theorem}\label{t:MR_2} 
Let the assumptions of Theorem \ref{t:F_deriv2} be satisfied. Assume moreover that for each $t\in[t_0,t_1]$
\begin{align}\label{uniform_conv_3}
\BE_t\big[|F(\xi^v)|\left(|C^v|+\varphi_\lambda(C^v)\right)\big]<\infty.
\end{align}
Then $f^v$ is continuously differentiable on $(t_0,t_1]$
with derivative given by
\begin{align}\label{e:derivative_4}
\frac{d}{dt}f^v(t)= t^{-1}\BE_t \int \nabla_x F(\xi^v)\,C^{v!}(dx).
\end{align}
\end{theorem}  
\begin{proof} Let $t>t_0$. Theorem \ref{t:F_deriv2} states that $f^v$
  is continuously differentiable on $(t_0,t_1]$ with derivative given
  by \eqref{e:derivative_2}. The assertion follows from
  \eqref{uniform_conv_3} and Lemma \ref{l:312}, since splitting $f$
into its negative and positive part we can apply Lemma \ref{l:312} to
  get
\begin{align*}
t \,\BE_t F(\xi^v) \varphi_\lambda(C^v)=\BE_t\int F(\xi^v-\delta_x)\,C^{v!}(dx).
\end{align*}
The result follows.
\end{proof}

\begin{remark}\label{r:MR}\rm Let the assumptions of 
Theorem \ref{t:MR_2} be satisfied.
By the Mecke equation \eqref{e:Mecke2v} we have 
\begin{align*}
\BE_t \int  \nabla_x F(\xi^v)\,C^{v!}(dx)
=t\,\BE_t \int  (F(\xi^{v,x})-F(\xi^{v}))
\I\{\text{$v \leftrightarrow x$ in $\xi^{v,x}$}\}\,\lambda(dx).
\end{align*}
If $v$ and $x$ are not connected in $\xi^{v,x}$, then
$F(\xi^{v,x})=F(\xi^{v})$. Therefore we can rewrite
\eqref{e:derivative_4} as
\begin{align}\label{e:derivative_7}
\frac{d}{dt}f^v(t)= \BE_t \int (F(\xi^{v,x})-F(\xi^v))\,\lambda(dx).
\end{align}
    \end{remark}

 \section{Differentiability of the cluster density}\label{s:diffenergy}
 
In this section we prove in particular that the position dependent cluster density
(given by \eqref{FE}) is continuously differentiable on $(0,\infty)$. 

\begin{theorem}\label{tdiffenergy} Suppose that $f\colon\N\to\R$ is a function satisfying 
\begin{align}\label{prop_f_1}
    \lim_{n\to\infty}f(n)\sqrt{n\log n}=0.
\end{align}
Then $t\mapsto \BE_t f(|C^v|)$ is for each $v\in\BX$ continuously differentiable on $(0,\infty)$ with derivative given by \eqref{e:derivative_2}.
\end{theorem} 

We prove the theorem via some lemmas, partially following
the proof of \cite[(LP) (3.6)]{BeGrimLoff98}. Let $v\in\BX$.
For $t>0$ and $n\in\N$ we define
 \begin{align*}
 p^v_{n}(t):= \BP_t(|C^v|=n).
 \end{align*}
Specializing definition \eqref{nun_f} in the case $f\equiv 1$ we set
\begin{align}\label{eboundnu}
\nu^v_n(\cdot):=\BE_1 \I\{\varphi_\lambda(C^v)\in \cdot\} \I\{ |C^v|=n\} 
e^{\varphi_\lambda(C^v)}.
\end{align}
Then we obtain from \eqref{f_n(t)} in the case $t_0=1$ that
\begin{align}\label{e672}
p^v_{n}(t)=t^{n-1}\int_0^\infty e^{-tu}\, \nu^v_n(du).
\end{align}
Since $p^v_1(t)=e^{-tD_\varphi (v)}$
we have
\begin{align}\label{nu1}
\nu^v_1=\delta_{D_\varphi(v)},\quad v\in\BX.
\end{align}

\begin{lemma}\label{l:nun1} Let $v\in\BX$, $n\in\N$ and $u>0$. Then
\begin{align}
    \nu^v_n[0,u]\le \left( \frac{e u}{n-1}\right)^{n-1},
\end{align}
where the right-hand side has to be interpreted as $1$ if $n=1$.
\end{lemma}
\begin{proof} In view of \eqref{nu1} we can assume that $n\ge 2$.
Since $p^v_{n}(t)\le 1$ for $t>0$, we have that
\begin{align*}
    t^{-n+1}\ge \int_0^u e^{-tu} \nu^v_n(d u)\ge e^{-t u} \nu^v_n[0,u).
\end{align*}
Optimizing over $t\in(0,\infty)$ yields the assertion.
\end{proof}

\begin{lemma}\label{l:nun3} Let $v\in\BX$, $n\in\N$ and $t>0$. Then
\begin{align*}
p_n^v(t)=t^n\int_0^\infty \nu^v_n[0,u]e^{-t u}\, du.
\end{align*}
\end{lemma}
\begin{proof}  The assertion follows from Lemma \ref{l:f_n_1}.
\end{proof}

\begin{lemma}\label{l:nun4} Let $v\in\BX$, $n\in\N$. Then
$t\mapsto p_n^v(t)$ is analytic on $(0,\infty)$ with derivative given by
\begin{align}\label{pndiff}
\frac{d}{dt}p_n^v(t)=nt^{n-1}\int_0^\infty \nu^v_n[0,u]e^{-t u}\, du
-t^n\int_0^\infty u \nu^v_n[0,u]e^{-t u}\, du
\end{align}
\end{lemma}
\begin{proof} 
The assertion follows from Lemma \ref{l:f_n_2}.
\end{proof}

Lemma \ref{l:nun4} implies 
\begin{align}\label{pndiffbound7}
\Big|\frac{d}{dt}p_n^v(t)\Big|\le \frac{n}{t}\int_0^\infty \nu^v_n[0, u] 
\left|1-\frac{ut}{n}\right|t^{n} e^{-t u}\,  du.
\end{align}
The next lemma provides a bound for the above right-hand side.

\begin{lemma}\label{l6.9} Let $v\in\BX$, $n\ge 2$ and $\delta\in(0,1)$. Then we have
for all $t>0$ that 
\begin{align}\label{pndiffbound}
\int_0^\infty \nu^v_n[0, u] \left|1-\frac{ut}{n}\right|t^{n} e^{-t u}  du
\le \delta p^v_{n}(t)
+ (1-\delta)^{n} e^{ \delta n}+(1+\delta)^{n} e^{-  \delta n}.
\end{align}
\end{lemma} 
\begin{proof} By Lemma \ref{l:nun3} we have 
\begin{align*}
\int_0^\infty \nu^v_n[0, u] \left|1-\frac{ut}{n}\right|t^{n} e^{-t u}  du
 & \le \delta p^v_{n}(t)
+ \int_{\left|1-\frac{tu}{n}\right|> \delta} \nu^v_n[0, u] t^{n} e^{-t u} \left|1-\frac{t u}{n}\right|d u \\
 & \le\delta p^v_{n}(t)+  \left(\frac{e}{n-1}\right)^{n-1}\int_{\left|1-\frac{tu}{n}\right|>\delta} t^nu^{n-1} 
e^{-t u} \left|1-\frac{t u}{n}\right|\,du.
\end{align*}
Changing variables yields that the above equals
\begin{align*}
\delta p^v_{n}(t)+ \left(\frac{e}{n-1}\right)^{n-1}\int_{\left|1-\frac{u}{n}\right|>\delta} u^{n-1} 
e^{-u} \left|1-\frac{u}{n}\right|\,du.
\end{align*}
Splitting the integral on the above right-hand side into two
pieces corresponding to $tu<n(1- \delta)$ and $t u>n(1+ \delta)$ yields
\begin{align*}
    \begin{split}
        \int_0^{n(1- \delta)} u^{n-1} e^{-u} \left(1-\frac{u}{n}\right) d u 
&=n^{n-1}(1-\delta )^ne^{-n(1-\delta )},\\
\int_{n(1+ \delta)}^\infty u^{n-1} e^{-u} \left(\frac{u}{n}-1\right) d u 
&=  n^{n-1}(1+\delta )^ne^{-n(1-\delta )}.
    \end{split}
\end{align*}
Since $(1+1/(n-1))^{n-1}<e$ for all $n\ge 2$, we obtain the assertion
\eqref{pndiffbound}.
\end{proof}

Let $f$ be as in Theorem \ref{tdiffenergy} and $v\in\BX$. 
Then $f$ is bounded and
\begin{align*}
\BE_t f(|C^v|)=\BE_t F(\xi^v)=\sum_{n=1}^\infty f(n) p^v_n(t).
\end{align*}
In order to prove Theorem \ref{tdiffenergy} we will check the condition \eqref{uniform_conv_2}
on $[t_0,\infty)$ for each $t_0>0$.
This is achieved by the previous and the following lemma.

\begin{lemma}
Suppose that $f\colon\N\to\R$. Then
    \begin{align}\label{e:4952}
        \BE_t|M^v_t F_n(\xi^v)|\le |f(n)|\int_0^\infty \nu^v_n[0, u] 
\left|n-ut\right|t^{n} e^{-t u} du.
    \end{align}
\end{lemma}
\begin{proof}
It is easy to see the following identities which follow from integration by parts
\begin{align*}
    \int_u^{(n-1)/t}(n-ts)e^{-ts} ds&=-t^{-1}(n-ts)e^{-ts}\Big|_u^{(n-1)/t}-\int_u^{(n-1)/t}e^{-ts} ds\\
    &=t^{-1}\left((n-tu)e^{-tu}\pm e^{-(n-1)}-e^{-tu}\right)=t^{-1}(n-1-tu)e^{-tu},\\
    \int_u^{\infty}(n-ts)e^{-ts} ds&=t^{-1}(n-tu)e^{-tu}-\int_u^{\infty}e^{-ts} ds=t^{-1}(n-1-tu)e^{-tu}.
\end{align*}
Since $\nu^v_n$ is locally finite, we can apply  Fubini's theorem to obtain that
    \begin{align*}
      &\BE_t|M^v_t F_n(\xi^v)|\\
      &=|f(n)| \BE_t\big[|n-1-t\varphi_\lambda(C^v)|\I\{|C^v|=n\}\big]= |f(n)| t^{n-1}\int_0^\infty |n-1-tu|e^{-tu}\nu^v_n(d u)\\
      &= |f(n)|t^{n-1}\left(\int_0^{(n-1)/t} (n-1-tu)e^{-tu}\nu^v_n(d u) - \int_{(n-1)/t}^\infty (n-1-tu)e^{-tu}\nu^v_n(d u)\right)\\
      &= |f(n)|t^{n}\left(\int_0^{(n-1)/t} \int_u^{(n-1)/t}(n-ts)e^{-ts} ds\,\nu^v_n(d u) - 
\int_{(n-1)/t}^\infty\int_u^\infty (n-ts)e^{-ts} ds \,\nu^v_n(d u)\right)\\
      &= |f(n)|t^{n}\left(\int_0^{(n-1)/t} \nu_n^v[0,s](n-ts)e^{-ts} ds - \int_{(n-1)/t}^\infty \nu_n^v[0,s] (n-ts)e^{-ts} ds \right)\\
      &= |f(n)|\int_0^\infty \nu^v_n[0, s] 
         \left|n-st\right|t^{n} e^{-t s}\,ds-2|f(n)|\int_{(n-1)/t}^{n/t} \nu_n^v[0,s](n-ts)t^ne^{-ts} ds.
    \end{align*}
\end{proof}

\begin{proof}[Proof of Theorem \ref{tdiffenergy}]
Let $v\in\BX$, $t_0>0$  and $n\ge 2$.
We need to check the condition \eqref{uniform_conv_2}.
To do so, we start with inequality \eqref{e:4952}.
In \eqref{pndiffbound} we choose
$\delta\equiv \delta_n$ by $\delta_n:=\sqrt{9\log n / n}$. 
We use the inequalities $(1-r)e^r\le e^{-r^2/2}$ which holds
for all $r\ge 0$ and $(1+r)e^{-r}\le e^{-r^2/3}$ which holds for all
$r\in[0,1/3)$. Then we obtain for  all sufficiently large  $n\in\N$
that $(1-\delta_n)^ne^{-n\delta_n}\le n^{-9/2}$ and
$(1-\delta_n)^ne^{-n\delta_n}\le n^{-3}$.
Hence there exist $n_0\in\N$ such that for each $t\ge t_0$
\begin{align}\label{e6.20}
    \sum_{n=n_0}^\infty t^{-1}|f(n)|\int_0^\infty \nu^v_n[0, u] 
\left|n-ut\right|t^{n} e^{-t u} du
\le \frac{\sqrt{9}}{t_0} \sum_{n=n_0}^\infty  |f(n)|\sqrt{n \log n} p^v_{n}(t)
+  \frac{2}{t_0}  \sum_{n=n_0}^\infty n^{-2}.
\end{align}
Let $\varepsilon>0$ and choose $n_1\ge n_0$ such that
$|f(n)|\sqrt{n \log n}\le \varepsilon$ for each $n\ge n_1$.
Then 
$$
\sum_{n=n_1}^\infty  |f(n)|\sqrt{n \log n} p^v_{n}(t)\le \varepsilon,
$$
finishing the proof.
\end{proof}

\begin{proof}[Proof of Theorem \ref{t:F_deriv3}] 
We check condition 
\eqref{uniform_conv_2}. By assumption \eqref{prop_f_7} it suffices to show that 
\begin{align*}
\lim_{n\to\infty}\sup_{t\ge t_0}\sum_{k>n} |\tilde{f}(k)|\BE_t\big|M^v_t\big| \I\{|C^v|=k\})=0,
\end{align*}
for any $t_0>0$. This follows from \eqref{e:4952} and the proof of Theorem \ref{tdiffenergy}.
\end{proof}

Later we shall need the following integrated version of
Theorem \ref{tdiffenergy}

\begin{theorem}\label{tdiffenergymarked} Assume that $\BX=\BY\times\BM$ is the product
of two complete separable metric spaces and 
let $\BQ$ be a finite measure $\BQ$ on $\BM$.  
Suppose that $f\colon\N\to\R$ is a function satisfying  \eqref{prop_f_1}.
Then $t\mapsto \int \BE_t f(|C^{(y,q)}|)\,\BQ(dq)$ is for each $y\in\BY$ continuously differentiable on $(0,\infty)$.
\end{theorem} 
\begin{proof} Let $y\in\BY$ and $t_0>0$. We know from Theorem \ref{tdiffenergy} that
$t\mapsto  \BE_t f(|C^{(y,q)}|)$ is for each $(y,q)\in\BY\times\BM$ continuously differentiable.
The assertion follows from the Leibniz differentiation rule once we can show that 
\begin{align}
\sum_{n=1}^\infty |f(n)|\int_0^\infty \nu^{(y,q)}_n[0, u] 
\left|n-ut\right|t^{n} e^{-t u} du  \le c,\quad t\ge t_0,\,q\in\BM,
\end{align} 
for some $c>0$.
Since $f(n)\sqrt{n\log n}$ is bounded, we see from \eqref{e6.20} that the above series, starting from
$n=n_0$ is bounded in $q\in\BM$ and $t\ge t_0$. The remaining terms in the series can be bounded by
\eqref{pndiffbound}. 
Similarly as in the proof of Lemma \ref{l:f_n_2} one can show that $\int f^{(y,q)}_n(t)\, \BQ(dq)$ is an analytic function on $(0,\infty)$. Therefore the continuity of the derivative follows from \eqref{e6.20}, since $\sum\limits_{n\ge n_0} |f(n)|\sqrt{n \log n} \int p_n^{(y,q)}(t)\,\BQ(d q)\to 0$ as $n_0\to\infty$ uniformly in $t\in \R_+$.
\end{proof}

Strengthening the assumption on the function $f$ in Theorem \ref{tdiffenergymarked},
we can write the derivative as a Margulis--Russo type formula.

\begin{theorem}\label{tdiff_MR} Suppose that $f\colon\N\to\R$ is a function satisfying 
\begin{align}\label{prop_f_2}
    \sup\limits_{n\ge1}|f(n)|n<\infty.
\end{align}
Then $t\mapsto \BE_t f(|C^v|)$ is for each $v\in\BX$ continuously differentiable on $(0,\infty)$ with derivative given by \eqref{e:derivative_4}.
\end{theorem} 
\begin{proof}
    It is enough to check condition \eqref{uniform_conv_3} on $[t_0,\infty)$ for each $t_0>0$. Condition \eqref{prop_f_2} implies that 
    \begin{align*}
        \BE_t |f(|C^v|)| |C^v|<\infty.
    \end{align*}
    It follows from Fubini's theorem and Lemma  \ref{l:nun1} that for $n\ge 2$ 
    \begin{align*}
        \BE_t |F_n(\xi^v)|\varphi_\lambda(C^v)&=|f(n)|t^{n-1} \int_0^\infty u e^{-tu} \nu_n^v(du)
        =|f(n)|t^{n-1} \int_0^\infty\int_u^\infty (ts-1) e^{-ts}\, ds\, \nu_n^v(du)\\
        &=|f(n)|t^{n-1} \int_0^\infty \nu_n^v[0,s] (ts-1) e^{-ts}\, ds < |f(n)|t^{n} \int_0^\infty \nu_n^v[0,s] s e^{-ts}\, ds \\
        &\le  |f(n)|\left( 2n p^v_n(t)+t^n\int_{2n}^\infty \nu_n^v[0,s] s e^{-ts}\, ds\right)\\
        &\le  |f(n)|\left(2n p^v_n(t)+t^{-1}\left(\frac{e}{n-1}\right)^{n-1}\int_{2n}^\infty u^n e^{-u}\, du\right)\\
        &=  |f(n)|2n p^v_n(t)+t^{-1}|f(n)|\left(\frac{e}{n-1}\right)^{n-1} n! \  \BP(X_n\le n),
    \end{align*}
where $X_n$ has a Poisson distribution with parameter $2n$. 
By assumption \eqref{prop_f_2} the sum over the first terms is converging.
By a rather elementary concentration
inequality (using a Chernoff bound argument) we have $\BP(X_n\le n)\le (2/e)^{n}$ for each $n\in\N$. Therefore
the sum over the second terms is converging too. 
\end{proof}

\begin{remark}\label{r:FE_MR}\rm
The position dependent cluster density satisfies the condition \eqref{prop_f_2} and its derivative 
can be represented by \eqref{e:derivative_4}, i.e.
\begin{align*}
\frac{d}{dt} \kappa^v(t)
&= t^{-1} \left( \BP_t(|C^v|<\infty)-\kappa^v(t)-\BE_t\int |C^v-\delta_x|^{-1} \, C^{v!}(dx)\right)\\
&= \BE_t \int (|C^v(\xi^{v,x})|^{-1}-|C^v|^{-1})\I\{x\in C^v(\xi^{v,x})\}\, \lambda(dx).
\end{align*}
\end{remark}

\section{Deletion stability of the  stationary marked RCM}\label{s:tolRCM}

In this section we consider the stationary marked  RCM as introduced
in Section \ref{s:stationaryRCM}.  Hence we take a Poisson process $\eta$ on $\R^d\times \BM$
with intensity measure $t\lambda=t\lambda_d\otimes \BQ$ and consider
the random connection model $\xi$ based on $\eta$ and
a fixed connection function $\varphi\colon (\R^d\times\BM)^2\to [0,1]$
satisfying \eqref{e:4.1} and \eqref{e:intmarked}.

\begin{theorem}\label{tmain1} The infinite clusters of a stationary
marked random connection model are deletion stable. 
\end{theorem}

Our proof of the theorem partially follows the seminal paper
\cite{Aizenman87}. It requires a significant extension of some of
the arguments in \cite{Jiang11} treating the Gilbert graph with deterministic radii; 
see Example \ref{ex:Boolean}.

We need to introduce some further notation.  For
$\mu\in\bG$ and $(x,p)\in V(\mu)$ we denote by
$N^0(x,p,\mu)$  
the number of clusters   
in $C^{(x,p)}(\mu)-\delta_{(x,p)}$. 
Hence $N^0(x,p,\mu)$ is the number
of clusters in $\mu-\delta_{(x,p)}$ which are connected in $\mu$ with $(x,p)$.   
We then define $N^+(x,p,\mu)$ similarly to $N^0(x,p,\mu)$, except that at most one
infinite cluster is counted, i.e.
\begin{align*}
    N^+(x,p,\mu):= N^0(x,p,\mu)-\I\{N^\infty(x,p,\mu)\ge 1\}(N^\infty(x,p,\mu)-1).
\end{align*}
Given $B\in\cB(\R^d)$ and a measure $\nu$ on $\R^d\times\BM$ it will be convenient to write
$\nu_B:=\nu_{B\times\BM}$ for the restriction of $\nu$ to  $B\times\BM$.

We fix some arbitrary $t_0>0$. It is then no restriction of generality
to assume that $t\in (0,t_0]$.  Let $(B_n)_{n\in\N}$ be an increasing
sequence of convex and compact sets with union $\R^d$.  Our proofs
require a specific {\em coupling} of the RCM $\xi$ with two random
graphs $\xi_{n,0}$ and $\xi_{n,+}$, $n\in\N$, according two different
boundary conditions: {\em free} and {\em wired}.  To this end we let
$\tilde\xi$ be a RCM based on a Poisson process $\tilde\eta$ with
intensity measure $t_0\lambda$. We can assume without loss of
generality that $\eta$ is $t/t_0$-thinning of $\tilde\eta$ (see
\cite[Corollary 5.9]{LastPenrose18}) and that $\xi$ is given as the
restriction $\tilde\xi[\eta]$ of $\tilde\xi$ to the vertices from
$\eta$. Let us first set
$\xi_n:=\tilde\xi[\eta_{B_n}+\tilde\eta_{B^c_n}]$. This is a RCM
driven by the Poisson process $\eta_{B_n}+\tilde\eta_{B^c_n}$ which
has intensity measure $t\lambda_{B_n}+t_0\lambda_{B^c_n}$. We define
$\xi_{n,0}$ as the restriction $\xi[\eta_{B_n}]=\xi_n[\eta_{B_n}]$ of
$\xi$ to $B_n\times\BM$. This is a RCM driven by $\eta_{B_n}$. We
  let $\xi_{n,+}$ be the random graph resulting from $\xi_n$ by
  connecting all points from $\tilde\eta_{B^c_n}$ to each other. Then
$\xi_{n,+}$ is a RCM driven by $\eta_{B_n}+\tilde\eta_{B^c_n}$
with a connection function which is equal to one for any pair of
  points from $(B^c_n\times\BM)^2$ and $\varphi$ otherwise. The
reader should keep in mind that $\xi_{n,+}$ is a very simple function
of the RCM $\xi_n$.  
An important property of
this coupling is that $\xi_{n,0}$ is a subgraph of $\xi$, while $\xi$
is a subgraph of $\xi_{n,+}$ (in fact of $\xi_n$).

For $(x,p)\in\eta_{B_n}$ we define  $C^{(x,p)}_{n,0}:=C^{(x,p)}(\xi_{n,0})$ and
$C^{(x,p)}_{n,+}:=C^{(x,p)}(\xi_{n,+})$ 
noting that 
$$
V(C^{(x,p)}_{n,+})=V(C^{(x,p)}_{n,0})+\I\{(x,p)\leftrightarrow \tilde\eta_{B^c_n} \text{in $\xi_{n}$}\}\tilde\eta_{B^c_n}.
$$
Note also that if $(x,p)\notin \eta_{B_n}$, then $C^{(x,p)}_{n,0}=C^{(x,p)}_{n,+}=0$. 
Note also that $C^{(x,p)}_{n,+}$ is infinite if and only if
$(x,p)$ is connected (in $\xi_n$) to $\tilde\eta_{B^c_n}$. Otherwise it is finite and coincides (by the coupling construction) with $C^{(x,p)}_{n,0}=C^{(x,p)}(\xi)$.

\begin{lemma}\label{l7.3} Let $(x,p)\in\R^d\times\BM$. Then, almost surely,
$V(C^{(x,p)}_{n,0})\uparrow V(C^{(x,p)}(\xi))$ and $V(C^{(x,p)}_{n,+})\downarrow V(C^{(x,p)}(\xi))$ as $n\to \infty$.
\end{lemma}
\begin{proof} 
Let $(x,p)\in\eta$ otherwise the statement is trivial.
There exists $m\in\N$ such that $x\in B_m$. 
We shall always take $n\ge m$.
The second assertion has to be interpreted this way.
Clearly $C^{(x,p)}_{n,0}$ is a subgraph of  $C^{(x,p)}(\xi)$.
Assume that $(y,q)\in C^{(x,p)}(\xi)$.
Then there exists $n$ such that $(x,p)$ is connected to  
$(y,q)$ within $\xi_{n,0}$. This proves the first assertion.
Next we note that 
$C^{(x,p)}(\xi)$ is a subgraph of  $C^{(x,p)}_{n+1,+}$ while
$C^{(x,p)}_{n+1,+}$ is a subgraph of $C^{(x,p)}_{n,+}$. 
Assume that $(y,q)\in C^{(x,p)}_{n,+}$ for each $n\ge m$. For large enough $n$ we then
have  $y\in B_{n}$ and hence $(y,q)\in C^{(x,p)}(\xi)$.
\end{proof}

For each $n\in\N$ we define
\begin{align*}
    M_{n,\star}:=\int |C^{(x,p)}_{n,\star}|^{-1}\,\eta_{B_n}(d(x,p)),
\end{align*}
where we use a star to denote either $0$ or $+$. 
A simple counting argument shows that $M_{n,0}$ is the number of clusters (finite)
in $\xi_{n,0}$ while $M_{n,+}$ is the number of finite clusters in $\xi_{n,+}$.
Moreover, we have that
\begin{align}\label{e7.35}
    M_{n,+}\le  M_n \le M_{n,0},
\end{align}
where (see also Proposition \ref{p:ergodiclimit})
\begin{align}
M_n:=\int |C^{(x,p)}(\xi)|^{-1}\,\eta_{B_n}(d(x,p)).
\end{align}

Recalling the definition \eqref{e:markedfree} of the cluster density $\kappa(t)$, we have the following lemma. 

\begin{lemma}\label{lapproximate}
Let $t\in[0,t_0]$. Then  $(\lambda_d(B_n))^{-1}\BE_t M_{n,\star}\to t\kappa(t)$ as $n\to\infty$. 
\end{lemma}
\begin{proof} By Lemma \ref{l:915},
\begin{align*}
\BE_t M_n&=t\iint \I\{x\in B_n\}\BE_t |C^{(x,p)}|^{-1}\,dx\,\BQ(dp)
=\lambda_d(B_n) t\kappa(t).
\end{align*}
Almost surely  $M_{n,0}-M_n$ is less than the number
of clusters with points from $\eta_{B_n}$ which are connected in $\xi$ with
$\eta_{B_n^c}$, and therefore less than the number of points from
$\eta_{B_n}$ which are directly connected in $\xi$ with $\eta_{B_n^c}$. Analogously,
$M_n-M_{n,+}$ is less than number of clusters with points from $\eta_{B_n}$ which are 
connected in $\xi_n$ with
$\tilde\eta_{B_n^c}$, and therefore less than the number of points from
$\eta_{B_n}$ which are directly connected in $\xi_n$ with $\tilde\eta_{B_n^c}$. Then with
probability one, we have
\begin{align*}
\begin{split}
  M_{n,0}- \int \I\{(x,p)\sim \tilde\eta_{B^c_n}\textrm{ in $\xi_n$}\} \eta_{B_n}(d (x,p))
&\le  M_{n,0}- \int \I\{(x,p)\sim \eta_{B^c_n}\textrm{ in $\xi$}\} \eta_{B_n}(d(x,p))\\
  \le  M_n &\le M_{n,+}+ \int \I\{(x,p)\sim \tilde\eta_{B^c_n}\textrm{ in $\xi_n$}\}  \eta_{B_n}(d(x,p)).
\end{split}
\end{align*}
By the Mecke equation, we have
\begin{align*}
    &\BE_t \int \I\{(x,p)\sim \tilde\eta_{B^c_n}\textrm{ in $\xi_n$}\}\eta_{B_n}(d (x,p))\\
&=t \iint \I\{x\in B_n\}\left(1-e^{-t_0\int\I\{y\in B^c_n\} \varphi(y-x,p,q)\, d y\,\BQ(dq)}\right)\, d x\,\BQ(dp) \\
 &   \le  t_0 t \iiint\I\{x\in B_n, y\in B^c_n\} \varphi(y-x,p,q)\, d x\, d y\,\BQ^2(d(p,q))\\
&=t_0 t \iint\I\{x\in B_n, y\in B^c_n\} \psi(y-x)\, d x\, d y,
\end{align*}
where 
\begin{align}
\psi(x):=\int \varphi(x,p,q)\,\,\BQ^2(d(p,q)),\quad x\in\R^d.
\end{align}
By assumption \eqref{e:intmarked},  $\psi$ is integrable.

Let $\varepsilon>0$ and choose $r>0$ so large that $\int \I\{|z|> r\} \psi(z)\, d z\le\varepsilon$. Then 
\begin{align*}
    \int \I\{x\in B_n, y\in B^c_n, |y-x|> r\} \psi(y-x)\, d x\,d y\le \varepsilon \lambda_d (B_n).
\end{align*}
Further
\begin{align}\label{c:amenability}
 \frac{1}{\lambda_d(B_n)}\int \I\{x\in B_n, y\in B^c_n, |y-x|\le r\} \psi(y-x)\, dx\, dy\le d_\varphi \frac{ \lambda_d((B_{n})_{\ominus r})}{\lambda_d(B_n)}
\overset{n\to\infty}{\longrightarrow}0,
\end{align}
where, for a bounded set $B\subset\R^d$,  $B_{\ominus r}:=\{x\in B: d(x,\partial B)\le r\}$ 
and $\partial B$ denotes the boundary of $B$.
Therefore 
\begin{align*}
 \limsup_{n\to\infty} \frac{\BE_t M_{n,0}}{  \lambda_d(B_n)}-\varepsilon t_0 t & \le t\kappa(t) \le
      \liminf_{n\to\infty} \frac{\BE_t M_{n,+}}{ \lambda_d(B_n)}+\varepsilon t_0 t
\end{align*}
Taking into account \eqref{e7.35},  this yields the assertion.
\end{proof}

Let $n\in\N$. We will now explore the derivatives of
$t\mapsto \BE_t M_{n,\star}$.  For $(x,p)\in B_n\times\BM$ we define
$N^\star_{n}(x,p):=N^0(x,p,\xi^{(x,p)}_{n,\star})$, the finite
volume counterparts of $N^0(x,p,\xi^{(x,p)})$ and  $N^+(x,p,\xi^{(x,p)})$.
By this definition  $N^0_n(x,p)$
is the number of (finite) clusters in $\xi_{n,0}$ which are connected
to $(x,p)$ in $\xi^{(x,p)}_{n,0}$, and $N^+_n(x,p)$ is the number of
clusters (with at most one infinite) in $\xi_{n,+}$ which are
connected to $(x,p)$ in $\xi^{(x,p)}_{n,+}$.

\begin{lemma}\label{l7.7} For any $n\in\N$ and either choice of boundary conditions
the function $t\mapsto \BE_t M_{n,\star}$ is differentiable on $[0,t_0)$ and the derivative is given by
\begin{align*}
\frac{d}{d t}\BE_t M_{n,\star}
=\lambda_d(B_n)-\BE_t \iint \I\{x\in B_n\}N^\star_n(x,p)\,dx\,\BQ(dp).
\end{align*}
\end{lemma}
\begin{proof} Since $M_{n,\star}\le \eta(B_n)$ we have 
$\BE_tM_{n,\star}<\infty$ for all $t>0$.  We now apply the Margulis-Russo formula \eqref{MargulisRusso}, where 
$\lambda_2=(\lambda_d)_{B_n}\otimes\BQ$ and $\lambda_1=0$ for the free boundary condition ($\star=0$) and $\lambda_1=t_0(\lambda_d)_{B^c_n}\otimes\BQ$ for the wired boundary condition ($\star=+$).
Hence
$\BE_t M_{n,\star}$ is a differentiable function of $t$ and
\begin{align*}
  \frac{d}{d t}\BE_t M_{n,\star}=\BE_t\iint\I\{x\in B_n\}\big(M_{n,\star}(\xi_{n,\star}^{(x,p)})-M_{n,\star}(\xi_{n,\star})\big)\,dx\,\BQ(dp).
\end{align*}

Let $(x,p)\in B_n\times\BM$.  If $N^\star_n(x,p)=0$,
then with probability one $M_{n,\star}(\xi_{n,\star}^{(x,p)})-M_{n,\star}(\xi_{n,\star})=1$.
Otherwise the removal of $(x,p)$ from $\xi^{(x,p)}_{n,\star}$ results in 
$M_{n,\star}(\xi_{n,\star}^{(x,p)})-M_{n,\star}(\xi_{n,\star})=1-N^\star_n(x,p)$ a.s.,
proving the result.
\end{proof}

\begin{lemma}\label{lconvex}
For any $n\in\N$ and either choice of boundary conditions
$\BE_t M_{n,\star} + \lambda_d(B_n)d_\varphi t^2/2$ is a convex function of $t$
on $[0,t_0)$.
\end{lemma}
\begin{proof} For $(x,p)\in\R^d\times\BM$
we let $\Psi(x,p)$ denote the point process of the {\em Poisson neighbours}
of $(x,p)$ in $\xi^{(x,p)}$, that is the points in $\eta$ which are
directly connected to $(x,p)$ in $\xi^{(x,p)}$.  For a Borel set $B\subset\R^d$ we let $\Psi_B(x,p)$
denote the restriction of $\Psi(x,p)$ to $B\times\BM$.
We further denote $d_\varphi(p):=\int d(p,q)\,\BQ(dq)$ so that
$d_\varphi=\int d_\varphi(p) \,\BQ(dp)$. By Lemma \ref{l7.7},
\begin{align*}
  &\frac{d}{d t}\bigg[\BE_t M_{n,\star}+\lambda_d(B_n)d_\varphi\frac{t^2}{2}\bigg]\\
&=\lambda_d(B_n)-\BE_t \iint \I\{x\in B_n\}N^\star_n(x,p)\,dx\,\BQ(dp)+t\lambda_d(B_n) d_\varphi\\
&=\lambda_d(B_n)+\iint \I\{x\in B_n\}\big(t d_\varphi(p)-\BE_t N^\star_n(x,p)\big)\,dx\,\BQ(dp)\\
&=\lambda_d(B_n)+\iint \I\{x\in B_n\}\big(\BE_t|\Psi(x,p)|- \BE_t N^\star_n(x,p)\big)\,dx\,\BQ(dp)\\
&=\lambda_d(B_n)+\iint \I\{x\in B_n\}\left( \BE_t|\Psi_{B^c_n}(x,p)|+\BE_t\big(|\Psi_{B_n}(x,p)|- N^\star_n(x,p)\big)\right)\,dx\,\BQ(dp),
\end{align*}
Clearly $\BE_t|\Psi_{B^c_n}(x,p)|$ is increasing in $t$.  
We shall now argue that
$\BE_t\big[\Psi_{B_n}(x,p)-N^\star_n(x,p)\big]$ is increasing in $t$.  
Applying the Margulis-Russo  formula \eqref{MargulisRusso} 
similarly as in the proof of Lemma \ref{l7.7}, we see that
it is sufficient to check that $\Psi_{B_n}(x,p)-N^\star_n(x,p)$
cannot strictly decrease when adding a point $(y,q)\in B_n\times\BM $ to $\eta$.
Assume first that $(y,q)$ is not directly connected to
$(x,p)$. Then $\Psi_{B_n}(x,p)$ does not change while $N^\star_n(x,p)$ can
only decrease (namely by connecting some of the clusters in
$\xi_{n,\star}$ which are connected to $(x,p)$ in
$\xi^{(x,p)}_{n,\star}$).  Assume now that $(y,q)$ is directly
connected to $(x,p)$, so that $\Psi_{B_n}(x,p)$ increases by one. In
that case $N^\star_n(x,p)$ can increase by at most $1$,
namely if some of the clusters in $\xi_{n,\star}$ which are not
connected to $(x,p)$ in $\xi^{(x,p)}_{n,\star}$ get connected to the
new point $(y,q)$ while none of the clusters in $\xi_{n,\star}$ which
are connected to $(x,p)$ in $\xi^{(x,p)}_{n,\star}$ are connected by
$(y,q)$.  This proves the asserted monotonicity and hence the
convexity assertion.
\end{proof}

Now we are in the position to prove the first main result in this section.

\begin{theorem}\label{t:convex} The function $t\mapsto t\kappa(t)+d_\varphi t^2/2$ is continuously differentiable 
on $(0,\infty)$, convex on $\R_+$ and right differentiable at zero.
\end{theorem}
\begin{proof} The first assertion follows from Theorem
  \ref{tdiffenergymarked} while the second follows from Lemmas
  \ref{lapproximate} and \ref{lconvex} and the (elementary) fact that
the limit of a sequence of convex functions is convex. The
    function is right differentiable at zero since $\kappa$ is a
    monotone function.
\end{proof}

In the final step of the proof of Theorem \ref{tmain1} we need to identify
the limits of the derivatives in Lemma \ref{l7.7}. 

\begin{lemma}\label{l7.9} Let $t\in[0,t_0]$. Then
 \begin{align*}
\liminf_{n\to\infty} \, (\lambda_d(B_n))^{-1} \BE_t \iint \I\{x\in B_n\}N^0_n(x,p)\,dx\,\BQ(dp)
&\ge \int \BE_t N^0(0,p,\xi^{(0,p)})\,\BQ(dp),\\
   \limsup_{n\to\infty} \, (\lambda_d(B_n))^{-1} 
   \BE_t \iint \I\{x\in B_n\}N^+_n(x,p)\,dx\,\BQ(dp)
&\le \int \BE_t N^+(0,p,\xi^{(0,p)})\,\BQ(dp).
\end{align*}
\end{lemma}
\begin{proof} Similarly as in the proof of Proposition \ref{p:ergodic} by stationarity, we have 
\begin{align*}
\int \BE_t N^\star(0,p,\xi^{(0,p)})\,\BQ(dp)
=(\lambda_d(B_n))^{-1}\BE_ t \iint \I\{x\in B_n\}N^\star(x,p,\xi^{(x,p)})\,dx\,\BQ(dp),\quad n\in\N.
\end{align*}
Hence our task is to show that $N^\star(x,p,\xi^{(x,p)})$ is well approximated by $N^\star_n(x,p)$.
For a given Borel set $B\subset \R^d$ and $(x,p)\in \R^d\times\BM$ we denote by $N^0_B(x,p)$
the number of clusters in $\xi$ 
to which the Poisson neighbors of $(x,p)$
in $B\times\BM$ belong. Note that a.s.
\begin{align*}
N^0_B(x,p)+N^0_{B^c}(x,p)\ge  N^0(x,p,\xi^{(x,p)}). 
\end{align*}
It is, moreover, easy to see that $N^0_n(x,p)\ge N^0_{B_n}(x,p)$ a.s.
for each $n\in\N$. It follows that a.s.
\begin{align}\label{e:2345}
N^0_n(x,p)\ge N^0(x,p,\xi^{(x,p)})-N^0_{B_n^c}(x,p).
\end{align}
Obviously $N^0_{B_n^c}(x,p)$ is dominated by the number
of points  from $\eta_{B_n^c}$ which are directly connected to $(x,p)$ in $\xi^{(x,p)}$.
Therefore
\begin{align*}
\BE_tN^0_{B_n^c}(x,p)\le t \iint \I\{y\in B^c_n\}\varphi(y-x,p,q)\,dy\,\BQ(dq).
\end{align*}
It now follows from \eqref{e:2345} and exactly as in the proof of Lemma \ref{lapproximate} that
for each $\varepsilon>0$
\begin{align*}
\liminf_{n\to\infty}\, (\lambda_d(B_n))^{-1} \BE_t \iint \I\{x\in B_n\}N^0_n(x,p)\,dx\,\BQ(dp)
&\ge \int \BE_t N^0(0,p,\xi^{(0,p)})\,\BQ(dp)-\varepsilon t.
\end{align*}

This implies the first asserted inequality. The second follows from
$N^+_n(x,p)\le N^+(x,p,\xi^{(x,p)})$ a.s..
\end{proof}

\begin{proof}[Proof of Theorem \ref{tmain1}]
The convex function in Theorem \ref{t:convex} is differentiable and approximated
by the differentiable convex functions 
$(\lambda_d(B_n))^{-1}\BE_t M_{n,\star}+d_\varphi t^2/2$; see Lemmas \ref{lapproximate}
and \ref{l7.7}.
A classical result from convex analysis (see \cite[Theorem 25.7]{Rock69}) implies that
\begin{align*}
\lim_{n\to\infty}(\lambda_d(B_n))^{-1}\frac{d}{dt}\BE_t M_{n,\star}=\frac{d}{dt} t\kappa(t).
\end{align*}
Therefore we obtain from Lemma \ref{l7.7} that the limit inferior in Lemma \ref{l7.9}
coincides with the limit superior. Hence Lemma \ref{l7.9} yields
\begin{align*}
\int \BE_t N^0(0,p,\xi^{(0,p)})\,\BQ(dp)\le \int \BE_t N^+(0,p,\xi^{(0,p)})\,\BQ(dp),
\end{align*}
or 
\begin{align*}
\int \BE_t (N^0(0,p,\xi^{(0,p)})-N^+(0,p,\xi^{(0,p)}))\,\BQ(dp)\le 0.
\end{align*}
Since 
\begin{align*}
N^0(0,p,\xi^{(0,p)})- N^+(0,p,\xi^{(0,p)})= \I\{N^\infty(0,p,\xi^{(0,p)})\ge 1\}(N^\infty(0,p,\xi^{(0,p)})-1),
\end{align*}
we obtain 
\begin{align}
\int \BP_t(N^\infty(0,p,\xi^{(0,p)})\ge 2)\,\BQ(dp)=0.
\end{align}
Using stationarity as in the proof of Lemma \ref{l:915}  we see, that this is equivalent to the assertion.
\end{proof}

The preceding proof yields the following corollary.

\begin{corollary}\label{c:energyderiv} The cluster density is continuously differentiable function and
\begin{align*}
\frac{d}{dt}(t\kappa(t))=1-\int \BE_t N^0(0,p,\xi^{(0,p)})\,\BQ(dp).
\end{align*}
\end{corollary}

\begin{remark}\label{r:amenability}\rm The convergence
on the right-hand side of \eqref{c:amenability} is crucial
for the proof of Lemma \ref{lapproximate}.
This {\em amenability} property of Euclidean space  
is also important for Lemma \ref{l7.9}. 
It might be possible to extend the methods of this section to establish
deletion stability for other amenable homogeneous spaces.
But we do not know, how to prove or disprove deletion stability in a non-amenable
situation, like the hyperbolic space.
\end{remark}

\begin{remark}\label{r:coupling}\rm Assume that $\varphi(x,p,q)=\tilde\varphi(\|x\|,p,q)$ and $t=1$,
for a measurable function $\tilde\varphi\colon [0,\infty)\times\BM\times\BM\to [0,1]$
which is decreasing and right-continuous in the first coordinate.
Using the notation at \eqref{e:defxi} and \eqref{e:Qm} we define
\begin{align}\label{e:weight}
W_{m,n}:=\frac{\|X_m-X_n\|}{\tilde\varphi^{-1}(Z_{m,n},Q_m,Q_n)},\quad m,n\in\N,
\end{align}
where $\tilde\varphi^{-1}(s,p,q):=\inf\{r\ge 0:\varphi(r,p,q)\le s\}$, $(s,p,q)\in[0,1]\times\BM\times\BM$.
Given $r>0$ we define a RCM $\xi_r$ with vertex set $\eta$ by connecting $X_m$ with $X_n$
if $W_{m,n}\le r$. Note that $W_{m,n}\le r$ if and only if
\begin{align*}
Z_{m,n}\le \tilde\varphi(r^{-1}\|X_m-X_n\|,Q_m,Q_n).
\end{align*}
Since $\sum^\infty_{n=1}\delta_{(r^{-1}X_n,Q_n)}$ is under $\BP_1$ a Poisson process
with intensity measure $r^d\lambda_d\otimes\BQ$, we hence have
\begin{align}
\BP_1(\xi_r\in\cdot)=\BP_{r^d}(\xi\in\cdot),\quad r>0,
\end{align}
i.e.\ a joint coupling of the RCMs with different intensity parameters. In the unmarked case this
construction can be found in  \cite[Example 1.3]{Alexander95b}.
\end{remark}

\begin{remark}\label{r:MST}\rm Consider the setting of Remark \ref{r:coupling}
and the complete graph with vertex set $\eta$. We can interpret the random variable
\eqref{e:weight} as {\em weight} of the edge between $(X_m,Q_m)$ and $(X_n,Q_n)$.
As in \cite{Alexander95} we define the associated {\em minimal spanning forest} $T$ as the
forest (a graph without cycles) with vertex set $\eta$ and an edge between 
$(X_m,Q_m)$ and $(X_n,Q_n)$ if there is no path between these points with weights strictly
less than $W_{m,n}$. In special cases it was observed in \cite{AldousSteele92,Alexander95,Alexander95b,BeGrimLoff98}
that there is a close relationship between the RCM $\xi_r$ and $T$. For instance it was proved in 
\cite{Alexander95} that the trees (clusters) of $T$ are all infinite and can only have one or two ends.
Two-ended trees $T$ can only occur if $r$ equals the percolation threshold in which case $T$ contains
all points of the infinite clusters (should they exist). It would be interesting to explore the consequences of deletion stability of $\xi_r$ for $T$.
\end{remark}

\section{The stationary marked RCM: irreducibility and uniqueness}\label{s:markedirrunique}

In this section we consider a  stationary marked RCM $\xi$ as
introduced in Section \ref{s:stationaryRCM}.
When combined with Theorem \ref{tmain1},  Theorem \ref{t:unique}
immediately yields the following result.

\begin{theorem}\label{t:markedunique} An irreducible stationary
marked random connection model can almost surely have 
at most one infinite cluster.
\end{theorem}

\begin{remark}\rm
Theorem \ref{t:markedunique} and Corollary \ref{c:necess} show that an irreducible stationary
marked RCM is $2$-indivisible. In particular this holds
at the critical intensity $t_c$. 
This provides some evidence for
the absence of doubly-infinite paths at criticality.
In fact, it is a common belief that in Euclidean space there
is no infinite cluster in the critical phase.
\end{remark}

We now present several examples, starting with the classical stationary RCM;
see Example \ref{ex:classical}.

\begin{example}\label{r:RCMunique} \rm 
By Theorem \ref{t:markedunique} and Proposition \ref{p:irredunmarked}
the (unmarked) stationary RCM can have at most one infinite component.
This generalizes \cite[Theorem 6.3]{MeesterRoy}, where it is assumed
that $\varphi(x)=\tilde\varphi(\|x\|)$, $x\in\R^d$, for a
decreasing function $\tilde\varphi\colon[0,\infty)\to [0,1]$. 
The proof there is based on an extension of
the approach from \cite{BurtonKeane89} to the continuum and is very different from ours. 
\end{example}

Next we treat the simple case, where the connection
factorizes; see also \cite[Section 1.2]{CaicedoDickson24}.

\begin{example}\label{r:factor} \rm Let $\psi\colon\R^d\to[0,1]$ be a symmetric
function with $0<m_\psi:=\int\psi(x)\,dx<\infty$ and let $K\colon\BM^2\to[0,1]$
be measurable and symmetric. Assume that
$\varphi(x,p,q)=\psi(x)K(p,q)$, $(x,p,q)\in\R^d\times\BM^2$. 
Then $d_\varphi(p,q)=m_\psi K(p,q)$, $(p,q)\in\BM^2$. 
By Theorem \ref{l:77}   
$\xi$ is irreducible if and only if 
\begin{align*}
\sup_{n\ge 1}K^n(p,q)>0,\quad \BQ^2\text{-a.e.\ $(p,q)\in\BM^2$}.
\end{align*}
\end{example}

We continue with the examples from Section \ref{s:stationaryRCM}.

\begin{example}\label{r:BMunique} \rm
Let us consider Example \ref{ex:Boolean}.
Assume that there exists $x_0\in\R^d$ (for instance $x_0=0$) 
such that $\BQ(\{L:B(x_0,\varepsilon_0)\subset K\})>0$,
where $B(x_0,\varepsilon)$ denotes the ball with center $x_0$ and radius $\varepsilon$.
Assume also that the function $V$ is increasing w.r.t.\ set inclusion and that
$V(K)>0$ if $K\ne \emptyset$. We show now how irreducibility of $\xi$ follows
from Corollary \ref{t:irredatom}. As point $p_0$ in \eqref{min3} we can take the ball  $B(x_0,\varepsilon_0)$,
while the set $A$ is chosen as $\{L\in \mathcal{C}^d:B(x_0,\varepsilon_0)\subset L\}$. Then $\BQ(A)>0$ and 
\eqref{min1} holds.
To check  \eqref{min3} we take $K\in\mathcal{C}^d$. Then
\begin{align*}
d_\varphi(B(x_0,\varepsilon_0),K)=\int \big(1-e^{-V(K\cap (B(x,\varepsilon_0))}\big)\,dx.
\end{align*}
By assumption on $V$ this is positive, since
\begin{align*}
\int \I\{K\cap B(x,\varepsilon_0)\ne\emptyset\}\,dx
=\int \I\{y\in B(x,\varepsilon_0)\}\,dx>0,
\end{align*}
where $y$ is some point from $K$.
We can now apply Theorem \ref{t:markedunique} to conclude that
the infinite cluster is unique. For the spherical Boolean model this result
can be found as Theorem 3.6 in \cite{MeesterRoy}. For general Boolean models
(i.e.\ $\varphi(x,K,L)=\I\{K\cap (L+x)\ne \emptyset\}$) the result seems to be new. 
\end{example}

\begin{example}\label{ex:weightedunique}\rm The weighted RCM from Example \ref{ex:weighted}
is irreducible by Corollary \ref{c:irredreal}. Indeed we have $d_\varphi(p,q)=m_\rho g(p,q)^{-1}$, which is positive and monotone.
By Theorem \ref{t:markedunique} the infinite cluster is unique.
This was asserted in \cite{GHMM22} without providing details of a proof.
A more detailed proof  in a special case (based on the approach in \cite{BurtonKeane89})
was given in \cite{JaMoe17}.
\end{example}

\begin{example}\label{ex:NN}\rm
Consider a stationary marked RCM with $\BM$ as the space of
all locally finite simple counting measures on $\R^d$. 
Let $\BQ$ be a distribution of a simple stationary point process
$\chi$ satisfying $\BQ\{0\}=0$. 
For $x\in\R^d$ and $p\in\BM$ let $d(x,p)$ be the distance
between $x$ and $p$. 
Similarly as in Example \ref{ex:weightedunique} we consider a connection function of the form
\begin{align*}
\varphi(x,p,q)=\rho(d(-x,p)^{-\alpha}d(x,q)^{-\alpha}\|x\|^d)
\end{align*}
for a decreasing function $\rho\colon [0,\infty)\to [0,1]$ such
that $m_\rho:=\int \rho(\|x\|^d)\,dx$ is positive and finite
and where $\alpha>0$ is a fixed parameter.
By stationarity,
\begin{align*}
\iint \varphi(x,p,q)\,\BQ^{2}(d(p,q))\,dx
&=\iint \rho(d(0,p)^{-\alpha}d(0,q)^{-\alpha}\|x\|^d)\,dx\,\BQ^{2}(d(p,q))\\
&=m_\rho\int \I\{d(0,p)<\infty,d(0,q)<\infty\}d(0,p)^{\alpha}d(0,q)^{\alpha}\,\BQ^{2}(d(p,q))\\
&=m_\rho \bigg(\int d(0,p)^{\alpha}\,\BQ(dp)\bigg)^2.
\end{align*}
To ensure \eqref{e:intmarked} we assume that $\int d(0,p)^{\alpha}\,\BQ(dp)<\infty$,
which is a rather weak assumption. 

Next we check that $\int d_\varphi(p,q)\,\BQ(dq)>0$ for $\BQ$-a.e.\ $p$, 
so that \eqref{e:minred} holds.
Fix $p\in\bM\setminus\{0\}$ such that $0\notin p$.  As above we have
\begin{align*}
\iint \varphi(x,p,q)\,\BQ(dq)\,dx
=\iint \rho(d(x,p)^{-\alpha}d(0,q)^{-\alpha}\|x\|^d)\,dx\,\BQ(dq).
\end{align*}
Moreover, since $0\notin p$ there exist $\varepsilon,c>0$ such 
that $d(x,p)\ge c$ for $\|x\|\le\varepsilon$. It follows that
\begin{align*}
\iint \varphi(x,p,q)\,dx\,\BQ(dq)
\ge \iint\I\{\|x\|\le\varepsilon\} \rho(c^{-\alpha}d(0,q)^{-\alpha}\|x\|^d)\,dx\,\BQ(dq).
\end{align*}
Assume for the sake of contradiction that the above outer integral vanishes
and take $q\in\bM\setminus\{0\}$ such that
$\int\I\{\|x\|\le\varepsilon\} \rho(c_1^{-\alpha}d(0,q)^{-\alpha}\|x\|^d)\,dx\,=0$.
However, since $m_\rho>0$ and $\rho$ is decreasing, 
$\rho(r)$ is positive for sufficiently small $r>0$. The resulting contradiction shows
that $\int d_\varphi(p,q)\,\BQ(dq)>0$.
The function $\varphi(x,p,\cdot)$ is for all
$(x,p)\in\R^d\times\BM$ non-decreasing with respect to the natural
partial ordering on $\BM$.  Therefore, if $\BQ$ is associated,
then Theorem \ref{t:irredPOP} implies that $\xi$ is irreducible.
For instance we might take $\BQ$ as the distribution
of a Poisson process; see e.g.\ \cite{LaSzYo20}. Hence Theorem \ref{t:markedunique} applies.
\end{example}

\noindent 
\textbf{Acknowledgement:}
This research was supported by the Deutsche Forschungsgemeinschaft (DFG, German
Research Foundation as part of the DFG priority program `Random Geometric Systems'
SPP 2265) under grant LA 965/11-1. The authors wish to thank Markus Heydenreich for several stimulating discussions.

\end{document}